\documentclass[10.5pt, a4paper]{amsart}
\usepackage{amssymb, amsmath, amscd, bm, amsthm, pdfpages, setspace, hyperref, mathtools, subcaption, graphicx, cite, xfrac, empheq}

\def\@Rref#1{\hbox{\rm \ref{#1}}}
\def\Rref#1{\@Rref{#1}}
\theoremstyle{plain}
\newtheorem{theorem}{Theorem}[section]
\newtheorem{proposition}[theorem]{Proposition}

\newtheorem{assumption}[theorem]{Assumption}
\newtheorem{lemma}[theorem]{Lemma}

\theoremstyle{definition}
\newtheorem{definition}{Definition}[section]
\newtheorem{example}[definition]{Example}
\newtheorem{remark}[definition]{Remark}

\newcommand{\re}{\mathop{\rm Re}\nolimits}
\newcommand{\diag}{\mathop{\rm diag}\nolimits}

\newcommand{\dom}{\mathop{\rm dom}}
\newcommand{\inc}{\mathop{\textrm{in}}\nolimits}
\newcommand{\out}{\mathop{\textrm{out}}\nolimits}
\begin{document}

\title[Stabilization of Infinite-dimensional Systems]{Stabilization of Infinite-dimensional Systems under Quantization and
	Packet Loss}

\thispagestyle{plain}

\author{Masashi Wakaiki}
\address{Graduate School of System Informatics, Kobe University, Nada, Kobe, Hyogo 657-8501, Japan}
 \email{wakaiki@ruby.kobe-u.ac.jp}
 \thanks{This work was supported in part by the Kenjiro Takayanagi Foundation and
 	JSPS KAKENHI Grant Numbers
 	20K14362 and 24K06866.}

\begin{abstract}
We study the problem of stabilizing infinite-dimensional systems with
input and output quantization.
The closed-loop system we consider is subject to
packet loss, whose average
duration is assumed to be bounded.
Given a bound on the initial state, we propose a design method
for dynamic quantizers with zoom parameters.
We show that the closed-loop state starting in a given region exponentially
converges to zero if 
bounds on quantization errors and packet-loss intervals
satisfy suitable conditions.
Since the norms of the operators
representing the system dynamics are used 
in the proposed quantizer design,
we also present methods for approximately computing the operator norms.
\end{abstract}


\maketitle

\section{Introduction}
\subsubsection*{Motivation and related studies}
In networked control systems, components such as actuators, sensors, and
controllers are connected through a network.
The presence of communication channels introduces
quantization errors and packet dropouts.
On the other hand, we often encounter systems which are described by
delay differential equations and partial differential equations (PDEs).
They are typical classes of infinite-dimensional systems, i.e.,
systems with infinite-dimensional state spaces.
This paper addresses a control problem involving 
three aspects: quantization, 
packet loss, and infinite-dimensionality.

It is well known that applying a static quantizer with
a finite number of values does not in general 
preserve global asymptotic stability of 
the closed-loop system.
In fact, when the input of the quantizer is outside the quantization region,
the quantization error may be large and lead to
the instability of the closed-loop system. 
In the quantization region, the error is small but non-zero. 
As a result, asymptotic convergence is not achieved in general.
To overcome the above difficulties, dynamic quantizers with 
adjustable zoom parameters have been introduced in \cite{Brockett2000,Liberzon2003Automatica} for finite-dimensional systems.
By adjusting the zoom parameter, asymptotic convergence can be achieved 
without quantizer saturation.
An overview  on 
feedback control under data-rate constraints can be found in \cite{Nair2007}.

Packets containing measurements and control inputs may be
dropped due to communication failure.
Non-malicious packet loss is typically modeled as a  stochastic process.
Another cause of packet loss is denial-of-service (DoS) attacks in the form of 
jamming and intentional packet-dropping.
Not all attackers launch such malicious packet loss based on
stochastic models.
To describe the uncertainty of DoS attacks,
a deterministic model of packet loss has been introduced in \cite{Persis2014,Persis2015}.
The problem of stabilizing finite-dimensional systems with
quantization and packet loss has
been actively investigated; 
see, e.g., \cite{Tatikonda2004noisy,You2010,Okano2014} for non-malicious packet loss
and \cite{Wakaiki2020DoS,Feng2020,Liu2022} for malicious packet loss. 
We refer to the survey articles \cite{Hespanha2007,Cetinkaya2019_Survey}
for further literature on control systems with packet loss.

Quantized control of
infinite-dimensional systems is an active and expanding research area.
A necessary and sufficient condition for infinite-dimensional 
and time-varying autoregressive moving average model to be asymptotically
stabilizable at a given data rate has been provided in 
\cite{Nair2000SCL}.
Control of time-delay systems with limited
information has been extensively studied, e.g., in \cite{Liberzon2006,Fridman2009,Persis2010,Liu2015}.
Based on operator semigroup theory,
exponential $L^p/\ell^p$-input-to-state stability, $1\leq p \leq \infty$, has been characterized for
abstract infinite-dimensional linear systems with sampled-data controllers in \cite{Logemann2013}.
This result in the case $p= \infty$ covers the closed-loop system
with static uniform quantization. 
Stabilization of multiple classes of PDE systems with 
static quantization has also been investigated; see
\cite{Tanwani2016, Espitia2017,Song2021} for first-order hyperbolic systems,
\cite{Selivanov2016PDE, Katz2022, Zheng2024} for second-order parabolic systems, and
\cite{Kang2023} for third-order parabolic systems.
Control methods for flexible structure systems with static input quantization 
have been proposed, e.g., in \cite{Gao2018, Wang2021, Zhao2021,Ji2023}.

Dynamic quantizers with zoom parameters have been
designed for first-order hyperbolic systems in \cite{Liberis2020}.
Several coding procedures have been numerically investigated for
a class of nonlinear distributed-parameter systems called sine-Gordon chains
in the input quantization case \cite{Andrievsky2019} and the output quantization case \cite{Andrievsky2023}.
The problem of state estimation has been studied both analytically and numerically
for semilinear wave equations with communication constraints 
in \cite{Dolgopolik2024}.
However,
to the best of our knowledge,
previous research has not systematically addressed
the design
of packet-loss resilient dynamic quantizers
for
PDE systems or abstract infinite-dimensional systems.

\subsubsection*{Problem description}
We consider abstract infinite-dimensional linear systems, especially
regular linear systems \cite{Weiss1989_regular,Weiss1994}.
This class of systems often appears in
the study of PDEs with boundary control and observation. For the details of
regular linear systems, we refer to
the survey papers \cite{Weiss2001, Tucsnak2014} and the book \cite{Staffans2005}.
We assume that the input and output spaces of the regular linear system
are finite dimensional.
In our setting,
the continuous-time plant is connected to 
the discrete-time controller through
a zero-order hold and 
a generalized sampler  as in \cite{Logemann2013}, and the closed-loop system
is exponentially stable in the case
without quantization or packet loss.

This paper addresses the problem of stabilizing
the infinite-dimensional  sampled-data system under input and output quantization and packet loss.
The control objective of this problem is that the closed-loop state
starting in a given region exponentially
converges to zero.
We use a deterministic model of packet loss
as in the existing studies \cite{Persis2014,Persis2015, 
	Wakaiki2020DoS,Feng2020,Liu2022} on DoS attacks.
Our primary focus is on 
packet loss in the sensor-to-controller channel,
but the scenario where
packet dropouts occur 
simultaneously in the sensor-to-controller and
controller-to-actuator channels is also discussed.

To solve the stabilization problem, 
we use the following 
discretization-based approach, as is commonly done for 
quantized control of finite-dimensional systems (see,  e.g.,  \cite{Bondarko2014,Liu2022}):
First, we construct the discretized plant  consisting of
the continuous-time plant, the zero-order hold, and the generalized sampler.
Second, we design dynamic quantizers for the discretized plant with packet loss. This second step
is the main part of our study.
Since  the plant state only at sampling times is considered in the second step,
we show in the last step that the plant state exponentially converges to zero 
between sampling times.

\subsubsection*{Comparisons and contributions}
The main contribution of this study is 
the design of packet-loss resilient 
dynamic quantizers for
infinite-dimensional systems.
There are several differences between this study and the 
existing studies \cite{Liberis2020,Andrievsky2019,Andrievsky2023,Dolgopolik2024}
on the design of dynamic quantizers for PDE systems.
First, the proposed quantizers are resilient to packet loss, which
is a feature not addressed in the existing studies.
Second, while the existing studies are limited to
either input or output quantization, we deal with both input and output quantization.
Finally, we consider
abstract infinite-dimensional systems,
rather than a specific class of PDEs, which
enables us to extend the trajectory-based approach  developed for
finite-dimensional systems in \cite{Wakaiki2020DoS}.
Therefore, the approach we take in this paper differs from the Lyapunov-based approaches used 
in the existing studies \cite{Liberis2020,Dolgopolik2024}.

In the finite-dimensional case \cite{Wakaiki2020DoS},
the norms of the matrices in the state and output equations are used
for the quantizer design.
For infinite-dimensional systems, the matrices are replaced by
operators. Hence, we give
an efficient formula for computing the norms of these operators.
Furthermore,  
the finite-dimensional case \cite{Wakaiki2020DoS}
focuses on
plants with output quantization and
controllers consisting of a Luenberger observer and a feedback gain,
whereas
here we consider plants with input and output quantization
and general strictly proper controllers.
For infinite-dimensional plants, 
the observer-based controller used in \cite{Wakaiki2020DoS}
inherits infinite-dimensionality unless model reduction techniques are applied, 
which
leads to implementation challenges.
To allow a broader range of controllers, including finite-dimensional ones, it is crucial to relax the structural restriction imposed on controllers.

To achieve the aforementioned contributions,
we first study the output feedback stabilization 
problem for infinite-dimensional discrete-time systems
with quantization and packet loss.
We assume that 
a bound on the initial state of the closed-loop system is known.
The proposed 
quantizers dynamically change the zoom parameters to ensure that 
the plant input and output are contained in the quantization regions and that 
the quantization errors converge to zero. 
When packet loss occurs,
the system becomes open-loop  temporarily. Hence,
the quantizers employ different update methods depending on whether packet loss is detected.
We show that the closed-loop system achieves 
exponential convergence
if the parameters of quantization coarseness and packet-loss duration satisfy
suitable conditions.

Next, we develop  methods for computing 
the operator norms
used in the design of the dynamic quantizers.
We prove that the norms of approximate operators on finite-dimensional spaces
converge to the norms of the operators on infinite-dimensional spaces under some assumptions.
To state the problem more precisely,
let $A$ be an bounded linear  operator on an infinite-dimensional space 
$X$, and
let $(A_N)_{N\in \mathbb{N}}$ be a sequence of approximate operators on finite-dimensional spaces
$X_N$.
We show that the norm 
sequence
$(\|A_N\|_{\mathcal{L}(X_N)})_{N \in \mathbb{N}}$ converges to $\|A\|_{\mathcal{L}(X)}$ as $N \to \infty$.
Note that the operator $A - A_N$ is not well-defined because
$X$ and $X_N$ are different.

Finally, 
we consider sampled-data regular linear systems with
input and output quantization and packet loss.
We show that exponential convergence at sampling times can be extended 
to that on $[0,\infty)$ for the sampled-data system.
The feedthrough term of the discretized plant is used
to design the dynamic quantizers. 
This term is a matrix, but the input-output operator of the regular linear system appears 
in its definition. 
We derive a series representation of the feedthrough 
matrix,
assuming that the semigroup generator
governing the state evolution of the uncontrolled system has
a spectral expansion with respect to a Riesz basis like
Riesz-spectral operators \cite[Section~3.2]{Curtain2020}
and diagonalizable operators \cite[Section~2.6]{Tucsnak2009}.

The proposed quantizer design builds on the approach described
in 
our position paper \cite{WakaikiASCC2022}, which
has presented a method for designing 
packet-loss resilient dynamic quantizers 
for infinite-dimensional discrete-time systems with
output quantization.
In this paper, we extend it to 
sampled-data systems with 
input and output quantization and also
establish
methods for the approximate computation of 
the operator norms used in the quantizer design.

\subsubsection*{Paper organization}
The remainder of this paper is organized as follows.
In Section~\ref{sec:Quantization}, we propose 
a design method of dynamic quantizers
for infinite-dimensional  discrete-time systems with packet dropouts.
Section~\ref{sec:norm_approximation} is devoted to the
approximate computation of the operator norms used in the quantizer design.
In Section~\ref{sec:regular_sys}, we examine the inter-sample behavior and the 
feedthrough term of sampled-data regular linear systems.
Section~\ref{sec:example} contains simulation results, and
concluding remarks are given in Section~\ref{sec:conclusion}.

\subsubsection*{Notation}
We denote by $\mathbb{N}_0$ the set of nonnegative integers and 
by $\mathbb{R}_+$ the set of nonnegative real numbers.
The complex conjugate of a complex number $\lambda$ is denoted by
$\overline{\lambda}$. 
For $\lambda_1,\dots,\lambda_n \in \mathbb{C}$, we use the notation
$\diag(\lambda_1,\dots,\lambda_n) \in \mathbb{C}^{n \times n}$ to denote
the diagonal matrix whose
$j$th diagonal element is equal to $\lambda_j$.
We denote
the $(j,\ell)$-entry of a matrix $P$
by $P_{j\ell}$.
The space of all functions from $\mathbb{N}_0$ to $\mathbb{C}^n$ is denoted by $F(\mathbb{N}_0,\mathbb{C}^n)$. 

Let $X$ and $Y$ be Banach spaces. 
We denote by $X \times Y$ the product space of $X$ and $Y$, equipped with
the norm 
\[
\left\|
\begin{bmatrix}
	x \\ y
\end{bmatrix}	
\right\|_{X\times Y} \coloneqq (\|x\|_X^p + \|y\|_Y^p)^{1/p}
\]
for some $p \in[1,\infty)$ or
\[
\left\|
\begin{bmatrix}
	x \\ y
\end{bmatrix}	
\right\|_{X\times Y} \coloneqq \max\{ \|x\|_X,\,
\|y\|_Y\}.
\]
The space of all bounded linear operators from $X$ to $Y$ is denoted by $\mathcal{L}(X,Y)$.
We write $\mathcal{L}(X) \coloneqq \mathcal{L}(X,X)$. For a linear operator $A$ from $X$ to $Y$,
the domain of $A$ is denoted by $\dom (A)$.
For a subset $S \subset X$ and a linear operator $A\colon \dom (A) \subset X \to Y$, 
we denote by $A|_S$
the restriction of $A$ to $S$. 
The spectrum and the 
resolvent set of a linear operator $A\colon \dom (A) \subset X \to X$ are denoted by
$\sigma(A)$ and $\varrho(A)$, respectively.
The spectral radius of $A \in \mathcal{L}(X)$ is denoted by $r(A)$. 
We denote 
by $X'$ the dual space of $X$ and
by $A'$ the dual operator of $A \in \mathcal{L}(X,Y)$.
Let $Z$ and $W$ be Hilbert spaces. For a densely-defined linear operator 
$A$ from $Z$
to $W$,
the Hilbert space adjoint of $A$
is denoted by $A^*$.

\section{Stabilization of 
	discrete-time systems under quantization and packet loss}
\label{sec:Quantization}
In this section, the design of dynamic quantizers resilient to
packet loss is investigated for 
infinite-dimensional discrete-time 
systems. First, the networked control
systems we consider is introduced. Next,
an update rule for the zoom parameters of 
the dynamic quantizers is proposed for controllers implementing a {\em zero compensation strategy},
by which
the input is set to zero when packet loss occurs.
We then prove that the state of the closed-loop system starting in a given region exponentially
converges to zero. We further present an analogous
result for controllers utilizing 
a {\em hold compensation strategy}, by which the last 
input is held in the event of packet loss.

\subsection{Networked control system}
\label{sec:cls}
\subsubsection{Plant and controller}
Let $X$, $X_{\rm c}$, $U$, and $Y$ be
nontrivial  real or complex Banach spaces.
We consider the following discrete-time plant:
\begin{equation}
	\label{eq:plant}
	\left\{
	\begin{aligned}
		x(k+1) &= Ax(k) + Bq_{\inc}(k),\\
		y(k) &= Cx(k) + Dq_{\inc}(k)  
	\end{aligned}
	\right.
\end{equation}
for $k \in \mathbb{N}_0$ with initial state $x(0) = x^0\in X$,
where 
$A \in \mathcal{L}(X)$, $B \in \mathcal{L}(U,X)$, $C \in \mathcal{L}(X,Y)$, 
and
$D \in \mathcal{L}(U,Y)$.
In \eqref{eq:plant},
$x(k) \in X$, $y(k) \in Y$, and 
$q_{\inc}(k) \in U$ are the state,
output, and input, respectively, of the plant at time $k$. 
The plant input $q_{\inc}(k)$ is
the quantized value of the controller output
$u(k)$ at time $k$.

Throughout Section~\ref{sec:Quantization}, we assume that 
$r(A) \geq 1$, which
simplifies the stability analysis 
in Section~\ref{sec:exp_conv_analysis}. If $r(A) < 1$, then the plant \eqref{eq:plant} with 
zero input $q_{\inc}(k) \equiv 0$ is exponentially stable, i.e.,
there exist $\Omega \geq 1$ and $\gamma \in (0,1)$ such that 
$\|x(k)\|_X \leq \Omega \gamma^k \|x^0\|_X$ for all $k \in \mathbb{N}_0$;
see, e.g., \cite[Lemma~2.2.11]{Tucsnak2009}. It is not our interest to consider
the scenario where the open-loop system is already stable.

In the absence of packet loss, we use the following controller:
\begin{equation}
	\label{eq:controller}
	\left\{
	\begin{aligned}
		x_{\rm c}(k+1)
		&= Px_{\rm c}(k) + Qq_{\out}(k),\\
		u(k) &= R x_{\rm c}(k) 
	\end{aligned}
	\right.
\end{equation}
for $k \in \mathbb{N}_0$ with initial state $x_{\rm c}(0) = x_{\rm c}^0 \in X_{\rm c}$,
where 
$P \in \mathcal{L}(X_{\rm c})$, 
$Q\in \mathcal{L}(Y, X_{\rm c})$, and 
$R \in \mathcal{L}(X_{\rm c},U)$.
In \eqref{eq:controller},
$x_{\rm c}(k) \in X_{\rm c}$, 
$u(k) \in U$, and 
$q_{\out}(k) \in Y$ are 
the state, output, and input, respectively, of the controller at time $k$.
The controller input $q_{\out}(k) $
is the quantized value of the plant 
output $y(k)$ at time $k$. 

In Sections~\ref{sec:cls}--\ref{sec:exp_conv_analysis_hold},
we  consider the scenario  where a packet for
the plant output may be dropped
but 
the plant input is always successfully transmitted.
In the closed-loop system,
the controller sends an acknowledgment packet to the sensor 
immediately when it receives the quantized plant output 
$q_{\out}$, and the acknowledgment is assumed to be successfully
transmitted to the sensor by the next sampling time.
The sensor can find that 
the packet containing $q_{\out}$ has been dropped
if it does not receive the acknowledgment.  This use of acknowledgment packets is common in quantized control with packet loss; see, e.g., \cite{Tatikonda2004noisy,You2010,Okano2014} 
for the case of non-malicious packet loss and 
\cite{Wakaiki2020DoS,Feng2020,Liu2022} for the case of malicious packet loss.
Furthermore, the controller sends a notification packet to the actuator when
a transmission failure occurs in the sensor-to-controller channel.
Figure~\ref{fig:closed_loop} illustrates the networked control 
system we consider in Sections~\ref{sec:cls}--\ref{sec:exp_conv_analysis_hold}.

\begin{figure}[tb]
	\centering
	\includegraphics[width = 8cm,clip]{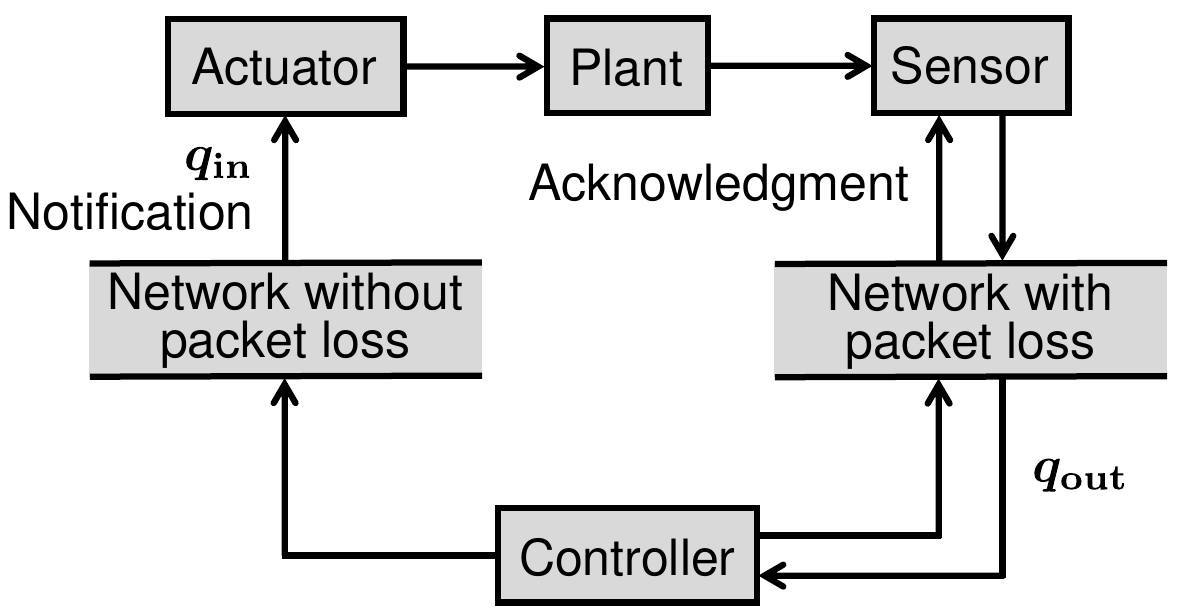}
	\caption{Networked control system with quantization and packet loss.}
	\label{fig:closed_loop}
\end{figure}

In the case without quantization or packet loss,
the plant state $x$ and 
the controller state $x_{\rm c}$ satisfy
\begin{equation}
	\label{eq:ideal_closed_loop}
	\begin{bmatrix}
		x(k+1) \\ x_{\rm c}(k+1)
	\end{bmatrix} = 
	\begin{bmatrix}
		A & BR  \\ QC & P+QDR
	\end{bmatrix}
	\begin{bmatrix}
		x(k) \\ x_{\rm c}(k)
	\end{bmatrix}
\end{equation}
for all $k \in \mathbb{N}_0$.
Define the state space $Z$ of the closed-loop system by $Z \coloneqq X \times X_{\rm c}$.
We assume that the ideal closed-loop system \eqref{eq:ideal_closed_loop}
is power stable. 
\begin{assumption}
	\label{assump:power_stability}
	The ideal closed-loop operator 
	\begin{equation}
		\label{eq:A_id_def}
		\mathcal{A}_{\rm id} \coloneqq
		\begin{bmatrix}
			A & BR  \\ QC & P+QDR
		\end{bmatrix}
	\end{equation}
	is power stable,
	i.e., there exist constants $M \geq 1$ and $\rho \in (0,1)$ such that 
	\begin{equation}
		\label{eq:Aid_power}
		\|\mathcal{A}_{\rm id}^k\|_{\mathcal{L}(Z)} \leq M \rho^k
	\end{equation}
	for all $k \in \mathbb{N}_0$.
\end{assumption}

\begin{remark}
	\label{rem:finite_dim_case}
	Let 
	the input space $U$ and the output space $Y$
	be finite dimensional.
	Assumption~\ref{assump:power_stability} is satisfied for
	a controller \eqref{eq:controller} 
	with finite-dimensional state space $X_{\rm c}$ if and only if
	the plant \eqref{eq:plant} is stabilizable
	and detectable;
	see \cite[Theorem~7]{Logemann1992}.
	The proof of this theorem is constructive, i.e., also provides a design method for 
	finite-dimensional stabilizing controllers. 
	\hspace*{\fill} $\triangle$ 
\end{remark}

\subsubsection{Packet loss}
To make an assumption on packet loss,
we define a binary variable $\theta(k) \in \{0,1\}$ by
\begin{equation*}
	\theta(k) \coloneqq 
	\begin{cases}
		1, & \text{if a packet dropout occurs at time $k$}, \\
		0, & \text{otherwise}
	\end{cases}
\end{equation*}	
for $k \in \mathbb{N}_0$.
Define $\Theta(k)$ by
\[
\Theta(k) \coloneqq \sum_{\ell=0}^{k-1} \theta(\ell)
\]
for $k \in \mathbb{N}$,
which denotes the total number of packet dropouts on the interval $[0,k)$.

Let us consider the situation where
packet loss is intentionally caused by malicious attackers.
High-power interference signals for jamming attacks 
demand energy expenditure, which may make continuous attacks impractical.
In addition, it is possible that attackers deliberately limit their offensive 
activities to evade detection. Such constraints
are formulated  in the following assumption introduced 
in \cite{Persis2014,Persis2015};
see also the survey \cite{Cetinkaya2019_Survey}.

\begin{assumption}
	\label{assump:duration}
	There exist constants $\Xi \geq 0$ and $\nu \in [0,1]$ such that 
	$\Theta(k)$ satisfies 
	\[
	\Theta(k) \leq \Xi + \nu k
	\]
	for all $k  \in \mathbb{N}$.
	We call $\nu$ the {\it duration bound} of packet loss.
\end{assumption}

We have $\limsup_{k \to \infty}
\Theta(k)/k \leq \nu$ under Assumption~\ref{assump:duration}. This implies 
that the time-averaged number of packet
dropouts is bounded.
The model in Assumption~\ref{assump:duration} is also closely related to
the deterministic model of non-malicious packet loss introduced in \cite{Hassibi1999CDC};
see also \cite{Hespanha2007,Persis2018}.

\subsubsection{Dynamic quantizer with zoom parameter}
\label{sec:encoding_scheme}
We use the following dynamic quantizers introduced in  \cite{Brockett2000,Liberzon2003Automatica}:
Let  $\Delta_{\inc},\Delta_{\out}\geq 0$.
Assume that the functions $\mathsf{Q}_{\inc} \colon U \to U$ 
and $\mathsf{Q}_{\out} \colon Y \to Y$
satisfy
\begin{subequations}
	\label{eq:quantization}
	\begin{align}
		\label{eq:quantization_input}
		\|u\|_U \leq 1\quad 
		&\Rightarrow \quad
		\| \mathsf{Q}_{\inc} (u) - u \|_U \leq \Delta_{\inc}, \\
		\label{eq:quantization_output}
		\|y\|_Y \leq 1\quad 
		&\Rightarrow \quad
		\| \mathsf{Q}_{\out} (y) - y \|_Y \leq \Delta_{\out}.
	\end{align}
\end{subequations}
The parameters $\Delta_{\inc}$ and $\Delta_{\out}$
are the error bounds for quantization. 
For $k \in \mathbb{N}_0$,
let $\mu_{\inc}(k),\mu_{\out}(k) >0$, and 
define the quantized values $q_{\inc}(k)$ and $q_{\out}(k)$ by 
\begin{align*}
	q_{\inc} (k) &\coloneqq \mu_{\inc}(k) \mathsf{Q}_{\inc}\left(  \frac{u(k)}{\mu_{\inc}(k)} \right), \\ 
	q_{\out} (k) &\coloneqq \mu_{\out}(k) \mathsf{Q}_{\out}\left(  \frac{y(k)}{\mu_{\out}(k)} \right).
\end{align*}
We call 
$\mu_{\inc}$ and $\mu_{\out}$
the {\em zoom parameters} for quantization.
By \eqref{eq:quantization}, for all $k \in \mathbb{N}_0$,
\begin{subequations}
	\label{eq:qe}
	\begin{align}
		\|u(k)\|_U \leq \mu_{\inc}(k)
		&\quad \Rightarrow \quad 
		\|q_{\inc}(k) - u(k)\|_U \leq  \Delta_{\inc}\mu_{\inc}(k) , \label{eq:qe_u} \\
		\|y(k)\|_Y \leq \mu_{\out}(k)
		&\quad \Rightarrow \quad 
		\|q_{\out}(k) - y(k)\|_Y \leq  \Delta_{\out}\mu_{\out}(k) . \label{eq:qe_y}
	\end{align}
\end{subequations}
Acknowledgments and packet-loss notifications allow us to use the binary variable
$\theta(k)$ to determine  $\mu_{\inc}(k+1)$ and $\mu_{\out}(k+1)$.

\subsection{Controller with zero compensation}
\label{sec:zero_compensation}
In this and the next subsection, we consider the following
controller equipped with the zero compensation strategy:
\begin{equation}
	\label{eq:controller_zero}
	\left\{
	\begin{aligned}
		x_{\rm c}(k+1) &= P_{\theta(k)}x_{\rm c}(k) + (1-\theta(k))
		Qq_{\out}(k),\\
		u(k) &= R x_{\rm c}(k) 
	\end{aligned}
	\right.
\end{equation}
for $k \in \mathbb{N}_0$ with initial state $x_{\rm c}(0) = x_{\rm c}^0 \in X_{\rm c}$,
where $P_0 = P$ and $P_1 \in \mathcal{L}(X_{\rm c})$ is a certain 
operator designed for the case when the packet containing
$q_{\out}$ is dropped.
A switching operator $P_{\theta(k)}$ has to be considered, for example,
in scenarios where 
an observer-based controller is implemented.

The closed-loop system can be written as
\begin{equation}
	\label{eq:closed_dynamics}
	\left\{
	\begin{aligned}
		z(k+1) &= 
		\mathcal{A}_{\theta(k)}z(k) + \mathcal{B}_{\inc,\theta(k)}(q_{\inc}(k) - u(k) )  + \mathcal{B}_{\out,\theta(k)}(q_{\out}(k) - y(k) ),\\
		u(k) &= \mathcal{C}_{\inc} z(k),   \\
		y(k) &= \mathcal{C}_{\out} z(k) + D(q_{\inc}(k) - u(k) ) 
	\end{aligned}
	\right.
\end{equation}
for all $k \in \mathbb{N}_0$,
where for $\theta \in \{0,1\}$,
\begin{subequations}
	\label{eq:zero_parameter}
	\begin{alignat}{2}
		z(k) &\coloneqq 
		\begin{bmatrix}
			x(k) \\ x_{\rm c}(k)
		\end{bmatrix},  & & \hspace{-70pt}
		z(0) = z^0 \coloneqq
		\begin{bmatrix}
			x^0 \\ x_{\rm c}^0
		\end{bmatrix}, \label{eq:z_def_zero} \\
		\mathcal{A}_{\theta} &\coloneqq
		\begin{bmatrix}
			A & BR \\ (1-\theta)QC & P_\theta+(1-\theta)QDR
		\end{bmatrix}, \label{eq:A_eta_def} &&\\
		\mathcal{B}_{\inc,\theta} &\coloneqq
		\begin{bmatrix}
			B \\ (1-\theta)QD
		\end{bmatrix},&&  \hspace{-70pt}
		\mathcal{B}_{\out,\theta} \coloneqq
		\begin{bmatrix}
			0 \\ (1-\theta)Q
		\end{bmatrix},\label{eq:B_eta_def}\\
		\mathcal{C}_{\inc} &\coloneqq
		\begin{bmatrix}
			0 & R
		\end{bmatrix},&& \hspace{-62.3pt}
		\mathcal{C}_{\out} \coloneqq
		\begin{bmatrix}
			C & DR
		\end{bmatrix}.\label{eq:C_eta_def}
	\end{alignat}
\end{subequations}
Note that 
$\mathcal{A}_0 = \mathcal{A}_{\rm id}$, where $\mathcal{A}_{\rm id}$ is 
the ideal closed-loop operator given in Assumption~\ref{assump:power_stability}.
We study the following notion of convergence
for the closed-loop system \eqref{eq:closed_dynamics}.
\begin{definition}
	\label{def:exp_conv}
	The closed-loop system \eqref{eq:closed_dynamics} 
	achieves {\em exponential convergence with initial state bound $E_0>0$} if
	there exist  constants 
	$\Omega \geq 1$ and $\gamma \in (0,1)$, independent of $E_0$,
	such that 
	\[
	\|z(k)\|_Z \leq \Omega  \gamma^k E_0
	\]
	for all $k \in \mathbb{N}_0$ and $z^0\in Z$ satisfying $\|z^0\|_Z \leq E_0$.
\end{definition}

\subsection{Exponential convergence under quantization and 
	packet loss}
\label{sec:exp_conv_analysis}
Let an initial state bound $E_0 >0$ be given, and 
let 
constants $M \geq 1$ and $\rho \in (0,1)$ be chosen such that 
\eqref{eq:Aid_power} holds for all $k \in \mathbb{N}_0$.
We update 
the zoom parameters $\mu_{\inc}$ and $\mu_{\out}$ in the following way:
\begin{equation}
	\label{eq:E_difference_eq}
	\left\{
	\begin{aligned}
		\mu(k+1) &\hspace{-0.92pt}=\hspace{-0.92pt} 
		\eta_{\theta(k)}\mu(k);\quad \mu(0) = ME_0,\\
		\mu_{\inc}(k) &\hspace{-0.92pt}=\hspace{-0.92pt}  \|R\|_{\mathcal{L}(X_{\rm c},U)} \mu(k),  \\
		\mu_{\out}(k) &\hspace{-0.92pt}=\hspace{-0.92pt}  (\|\mathcal{C}_{\out}\|_{\mathcal{L}(Z,Y)}  \!+\! 
		\Delta_{\inc} \hspace{-0.65pt}\|D\|_{\mathcal{L}(U,Y)} \hspace{-0.7pt} \|R\|_{\mathcal{L}(X_{\rm c},U)}
		)
		\mu(k)
	\end{aligned}
	\right.
\end{equation}
for $k \in \mathbb{N}_0$,
where  $\eta_0= \eta_0(\Delta_{\inc},\Delta_{\out}) $
and $\eta_1 = \eta_1(\Delta_{\inc}) $ are defined by
\begin{subequations}
	\label{eq:eta_def_subeq}
	\begin{align}
		\eta_0 
		&\coloneqq \rho+ M ( 
		\Delta_{\inc} \|\mathcal{B}_{\inc,0}\|_{\mathcal{L}(U,Z)} \hspace{1pt} 
		\|R\|_{\mathcal{L}(X_{\rm c},U)}+  
		\Delta_{\out}
		\|Q\|_{\mathcal{L}(Y,X_{\rm c})} \hspace{1pt}\|\mathcal{C}_{\out}\|_{\mathcal{L}(Z,Y)} 
		\label{eq:eta0_def} \notag\\
		&\hspace{58pt}+  \Delta_{\inc} \Delta_{\out}
		\|Q\|_{\mathcal{L}(Y,X_{\rm c})} \hspace{1pt}  
		\|D\|_{\mathcal{L}(U,Y)} \hspace{1pt}\|R\|_{\mathcal{L}(X_{\rm c},U)}
		), \\
		\eta_1 
		&\coloneqq M ( \|\mathcal{A}_1\|_{\mathcal{L}(Z)} 
		+ \Delta_{\inc} \|B\|_{\mathcal{L}(U,X)} \hspace{1pt} \|R\|_{\mathcal{L}(X_{\rm c},U)}
		). \label{eq:eta1_def}
	\end{align}
\end{subequations}
In the above update rule,
$\eta_0$ and $\eta_1$ represent the rates of 
change of $\mu$ in the absence and presence of packet loss, respectively.
Since 
$\eta_1 \geq r(A) \geq 1$, it follows that 
$\mu$ increases when packet loss occurs.

The following theorem shows that 
the closed-loop system 
\eqref{eq:closed_dynamics} 
achieves exponential convergence if
the quantization error bounds and the duration
bound of packet loss are small enough 
that the following two properties
are satisfied: (i)
$\mu$ decreases in the absence of packet loss; and
(ii) the duration of packet loss is such that $\mu$ decreases on average.
\begin{theorem}
	\label{thm:exp_conv}
	Suppose that 
	Assumptions~\ref{assump:power_stability} and \ref{assump:duration} hold.
	Let the state and the operators of the closed-loop system 
	\eqref{eq:closed_dynamics} 
	be given by \eqref{eq:zero_parameter}.
	Given an initial state bound $E_0 >0$, let the zoom parameters
	$\mu_{\inc}$ and $\mu_{\out}$ be defined as in \eqref{eq:E_difference_eq}.
	Assume that the quantization error bounds $\Delta_{\inc}, \Delta_{\out}\geq0$ and the duration
	bound $\nu \in [0,1]$
	of packet loss  satisfy
	\begin{align}
		\label{eq:ex_conv_cond}
		\eta_0  < 1 \quad \text{and} \quad 
		\nu < \frac{\log (1/\eta_0  )}{ \log(\eta_1 / \eta_0)},
	\end{align}
	where $\eta_0= \eta_0(\Delta_{\inc},\Delta_{\out}) $
	and $\eta_1 = \eta_1(\Delta_{\inc}) $ are defined by
	\eqref{eq:eta_def_subeq}.
	Then the following statements hold:
	\begin{enumerate}
		\renewcommand{\labelenumi}{\alph{enumi})}
		\item
		The closed-loop state $z$ satisfies 
		\[
		\|z(k)\|_Z \leq \mu(k)
		\] 
		for all $k \in \mathbb{N}_0$.
		\item
		There exist constants 
		$\Omega_{\mu} \geq 1$ and $\gamma \in (0,1)$, independent of $E_0$, such that 
		\[
		\mu(k) \leq \Omega_{\mu} \gamma^k \mu(0)
		\] 
		for all $k \in \mathbb{N}_0$. 
	\end{enumerate}
	In particular, the closed-loop system \eqref{eq:closed_dynamics} 
	achieves exponential convergence with initial state bound $E_0$.
\end{theorem}
\begin{proof}
	The last statement on the exponential convergence 
	property of the closed-loop system immediately follows from combining statements~a) and b). In the following,
	we will prove statements~a) and b).
	Assumption~\ref{assump:power_stability} is necessary
	for the first condition in \eqref{eq:ex_conv_cond}, i.e., $\eta_0<1$,
	to hold.
	The conditions in \eqref{eq:ex_conv_cond} 
	and Assumption~\ref{assump:duration}
	are used only for statement~b).
	
	Proof of statement~a).
	We introduce the following function $|\cdot|_Z\colon Z \to \mathbb{R}_+$, as in the proofs of
	\cite[Lemma~II.1.5]{Eisner2010} and \cite[Theorem~5.8]{Wakaiki2018_EVC}:
	\[
	|z|_Z \coloneqq \sup_{k \in \mathbb{N}_0} \|\rho^{-k}\mathcal{A}_{0}^k z\|_Z,\quad z \in Z.
	\]
	It is a norm on $Z$ and satisfies 
	\begin{subequations}
		\label{eq:new_norm_prop}
		\begin{gather}
			\|z\|_Z \leq |z|_Z \leq M \|z\|_Z, \label{eq:new_norm_prop1}\\
			|\mathcal{A}_0 z|_Z \leq \rho |z|_Z.\label{eq:new_norm_prop2}
		\end{gather}
	\end{subequations}
	for all $z \in Z$.
	These properties
	can be proved in the same way as in the finite-dimensional 
	case~\cite[Lemma~3.3]{Wakaiki2023}.
	
	Let the initial value $z^0 \in Z$
	of the closed-loop state
	satisfy $\|z^0\|_Z \leq E_0$.
	By the property \eqref{eq:new_norm_prop1} of the norm $|\cdot|_Z$,
	it is enough to prove that 
	\begin{equation}
		\label{eq:z_bound}
		|z(k)|_Z \leq \mu(k)
	\end{equation} 
	for all $k \in \mathbb{N}_0$.
	Since the property \eqref{eq:new_norm_prop1} gives 
	\[
	|z^0|_Z \leq M\|z^0\|_Z \leq ME_0 = \mu(0),
	\] 
	it follows that 
	\eqref{eq:z_bound} is true for $k=0$.
	
	We now proceed by induction and assume  that
	\eqref{eq:z_bound} is true  for some $k \in \mathbb{N}_0$.  
	In order to use \eqref{eq:qe},
	we will show that $\|u(k)\|_{U}\leq \mu_{\inc}(k)$ and
	$\|y(k)\|_{Y}\leq \mu_{\out}(k)$.
	By the second equation of \eqref{eq:closed_dynamics}, 
	\begin{equation*}
		\|u(k)\|_{U} \leq \|R\|_{\mathcal{L}(X_{\rm c},U)}
		\hspace{1pt} \|z(k)\|_Z.
	\end{equation*}
	Since 
	$\|z(k)\|_Z \leq \mu(k)$ by the inductive hypothesis and 
	the property \eqref{eq:new_norm_prop1},
	it follows that 
	$
	\|u(k)\|_{U}\leq  \mu_{\inc}(k).
	$
	Combining this and \eqref{eq:qe_u} with
	the third equation of  \eqref{eq:closed_dynamics}, we derive
	\begin{equation*}
		\|y(k)\|_{Y}\leq 
		\|\mathcal{C}_{\out}\|_{\mathcal{L}(Z,Y)} \hspace{1pt} \|z(k)\|_Z +  \Delta_{\inc} \|D\|_{\mathcal{L}(U,Y)} \mu_{\inc}(k).
	\end{equation*}
	Therefore,
	$
	\|y(k)\|_{Y} \leq  \mu_{\out}(k)
	$ holds.
	
	Let us 
	consider the case $\theta(k) = 0$.
	Applying the 
	properties \eqref{eq:new_norm_prop1} and  \eqref{eq:new_norm_prop2} to the first 
	equation of  \eqref{eq:closed_dynamics}, we obtain
	\begin{align*}
		|z(k+1)|_Z &\leq \rho |z(k)|_Z + 
		M\|\mathcal{B}_{\inc,0}\|_{\mathcal{L}(U,Z)} \hspace{1pt} 
		\|q_{\inc}(k) - u(k)\|_{U}  + 
		M\|\mathcal{B}_{\out,0}\|_{\mathcal{L}(Y,Z)} \hspace{1pt} 
		\|q_{\out}(k) - y(k)\|_{Y}.
	\end{align*}
	Hence, \eqref{eq:qe} yields
	\[
	|z(k+1)|_Z \leq 
	\eta_{0}
	\mu(k) = \mu(k+1).
	\]

	Now we consider the case $\theta(k) = 1$. By the first equation of  \eqref{eq:closed_dynamics} and the property~\eqref{eq:new_norm_prop1},
	we have
	\begin{align*}
		|z(k+1)|_Z &\leq  M\|\mathcal{A}_1\|_{\mathcal{L}(Z)}\hspace{1pt} |z(k)|_Z   + M\|\mathcal{B}_{\inc,1}\|_{\mathcal{L}(U,Z)}  \hspace{1pt} 
		\|q_{\inc}(k) - u(k)\|_{U}.
	\end{align*}
	This together with \eqref{eq:qe_u} 
	gives
	\[
	|z(k+1)|_Z  \leq \eta_{1} \mu(k) = \mu(k+1).
	\]

	Proof of statement~b).
	Using
	$\eta_0 < 1 \leq  \eta_1$,
	we obtain $\eta_1/\eta_0 > 1$.
	By Assumption~\ref{assump:duration}, 
	\begin{align*}
		\mu(k) = \eta_1^{\Theta(k)} \eta_0^{k - \Theta(k)} \mu(0) 
		\leq \left( \frac{\eta_1}{\eta_0} \right)^{\Xi}  \left( 
		\left( \frac{\eta_1}{\eta_0} \right)^{\nu} \eta_0
		\right)^k \mu(0)
	\end{align*}
	for all $k \in \mathbb{N}$.
	Since
	\[
	\left( \frac{\eta_1}{\eta_0} \right)^{\nu} \eta_0 < 1
	\quad 
	\Leftrightarrow
	\quad 
	\nu < \frac{\log (1/\eta_0  )}{ \log(\eta_1 / \eta_0)},
	\]
	it follows that statement~b) holds. 
\end{proof}

\begin{remark}
	The convergence of the zoom parameters to zero may not be practical 
	(e.g., due to computational limitations).
	In such a case, the zoom parameters need to be lower-bounded by a strictly 
	positive constant;
	see \cite[Remark~3]{Fradkov2009} and \cite{Fradkov2010}.
	The proposed update rule can be easily  modified as follows:
	Let $\delta >0$, and define 
	\begin{equation*}
		\mu(k+1) \coloneqq \eta_{\theta(k)} \mu(k) + \delta
	\end{equation*}
	for $k \in \mathbb{N}$. 
	Statement~a) follows from the same argument.
	To obtain an upper bound on  $\mu$ as in statement~b), 
	we define $\Theta(k_1,k_2)$ by
	\[
	\Theta(k_1,k_2) \coloneqq \sum_{\ell=k_1}^{k_2-1} \theta(\ell)
	\]
	for $k_1,k_2 \in \mathbb{N}_0$ with $k_1 < k_2$, and make 
	the following slightly
	stronger assumption on packet loss 
	than Assumption~\ref{assump:duration}:
	There exist $\Xi \geq 0$ and $\nu \in [0,1]$ such that 
	$\Theta(k_1,k_2)$ satisfies
	\[
	\Theta(k_1,k_2) \leq \Xi + \nu (k_2-k_1)
	\]
	for all $k_1,k_2 \in \mathbb{N}_0$ with $k_1 < k_2$.
	Then a routine calculation shows that  	for all $k \in \mathbb{N}$,
	\begin{align*}
		\mu(k) 
		&\leq 
		\Omega_{\mu} \gamma^k \mu(0) + \left(
		1+\frac{\Omega_{\mu }\gamma}{1-\gamma}
		\right)\delta,
	\end{align*}
	where $\Omega_{\mu} \coloneqq (\eta_1/\eta_0)^{\Xi}$ and
	$\gamma \coloneqq (\eta_1/\eta_0)^{\nu} \eta_0$. 
	\hspace*{\fill} $\triangle$ 
\end{remark}

\subsection{Convergence result for hold compensation strategy}
\label{sec:exp_conv_analysis_hold}
We present an analogous convergence result 
for controllers implementing
the hold compensation strategy.
Consider the controller given by
\begin{equation}
	\label{eq:controller_hold_comp}
	\left\{
	\begin{aligned}
		\begin{bmatrix}
			x_{\rm c}(k+1) \\ \hat q_{\rm \out}(k+1)
		\end{bmatrix}
		&= 
		\begin{bmatrix}
			P & \theta(k)Q \\
			0 & \theta(k) I 
		\end{bmatrix} 
		\begin{bmatrix}
			x_{\rm c}(k) \\ \hat q_{\rm \out}(k) 
		\end{bmatrix}  + 
		(1-\theta(k))
		\begin{bmatrix}
			Q \\ I
		\end{bmatrix}
		q_{\out}(k) ,\\
		u(k) &= 
		\begin{bmatrix}
			R & 0
		\end{bmatrix}
		\begin{bmatrix}
			x_{\rm c}(k) \\ \hat q_{\rm \out}(k) 
		\end{bmatrix}
	\end{aligned}
	\right.
\end{equation}
for $k \in \mathbb{N}_0$,
where 
$x_{\rm c}(0) = x^0_{\rm c} \in X_{\rm c}$ and 
$\hat q_{\rm \out}(0) = \hat q_{\rm \out}^{\,0} \in Y$.
The auxiliary  state $\hat q_{\rm \out} (k)\in Y$ represents the last received $q_{\rm \out}$
that the controller keeps at time $k$.
Note that 
we do not have to use a switching operator $P_{\theta(k)}$ 
for the hold compensation strategy even
when the controller contains an observer.

Define $Z \coloneqq X \times X_{\rm c} \times Y$.
The closed-loop system can be written in the form \eqref{eq:closed_dynamics} for $k \in \mathbb{N}_0$, where
for $\theta \in \{0,1\}$,
\begin{subequations}
	\label{eq:hold_parameter}
	\begin{alignat}{2}
		z(k) &\coloneqq 
		\begin{bmatrix}
			x(k) \\ x_{\rm c}(k) \\ \hat q_{\rm \out}(k) 
		\end{bmatrix}, & & \hspace{-90pt}z(0) = z^0 \coloneqq 
		\begin{bmatrix}
			x^0 \\ x_{\rm c}^0 \\ \hat q_{\rm \out}^{\,0}
		\end{bmatrix},\\
		\mathcal{A}_{\theta} &\coloneqq
		\begin{bmatrix}
			A & BR & 0 \\ (1-\theta)QC & P+(1-\theta)QDR & \theta Q \\
			(1-\theta)C &  (1-\theta)DR &  \theta I
		\end{bmatrix}, \label{eq:A_eta_def2} &&\\
		\mathcal{B}_{\inc,\theta} &\coloneqq
		\begin{bmatrix}
			B \\ (1-\theta)QD \\ (1-\theta)D 
		\end{bmatrix},& & \hspace{-90pt}
		\mathcal{B}_{\out,\theta} \coloneqq
		\begin{bmatrix}
			0 \\ (1-\theta)Q \\ (1-\theta)I
		\end{bmatrix},\label{eq:B_eta_def2}\\
		\mathcal{C}_{\inc} &\coloneqq
		\begin{bmatrix}
			0 & R & 0
		\end{bmatrix}, && \hspace{-82.2pt}
		\mathcal{C}_{\out} \coloneqq
		\begin{bmatrix}
			C & DR & 0
		\end{bmatrix}.\label{eq:C_eta_def2}
	\end{alignat}
\end{subequations}	
If Assumption~\ref{assump:power_stability} holds, then
the operator 
$\mathcal{A}_0$ defined by \eqref{eq:A_eta_def2} 
is power stable with the same decay rate $\rho \in (0,1)$ as 
the ideal closed-loop operator $\mathcal{A}_{\rm id}$.
Let $M_{\rm h} \geq 1$ satisfy
\[
\|\mathcal{A}_0^k\| \leq M_{\rm h} \rho^k
\] 
for all $k \in \mathbb{N}_0$.
The following result can be proved in the same way
as Theorem~\ref{thm:exp_conv}, and therefore we omit the proof.
\begin{theorem}
	\label{thm:exp_conv2}
	Suppose that 
	Assumptions~\ref{assump:power_stability} and \ref{assump:duration} hold. 
	Let the state and the operators of the closed-loop system 
	\eqref{eq:closed_dynamics} 
	be given by \eqref{eq:hold_parameter}. Given an initial state bound $E_0 >0$,
	let the zoom parameters
	$\mu_{\inc}$ and $\mu_{\out}$ be defined as  in \eqref{eq:E_difference_eq} with 
	$\|Q\|_{\mathcal{L}(Y,X_{\rm c})}$ replaced by 
	\[
	\left\|
	\begin{bmatrix}
		Q \\ I
	\end{bmatrix}
	\right\|_{\mathcal{L}(Y,X_{\rm c}\times Y)}
	\]
	and $M$ replaced by $M_{\rm h}$.
	Assume that the quantization error bounds $\Delta_{\inc}, \Delta_{\out}\geq 0$ and 
	the duration bound $\nu \in [0,1]$  of packet loss 
	satisfy the conditions given in  \eqref{eq:ex_conv_cond}.
	Then the conclusions of Theorem~\ref{thm:exp_conv} hold.
\end{theorem}

\subsection{
	Simultaneous packet loss in
	sensor-to-controller and controller-to-actuator channels}
\label{sec:sim_packet}
In Sections~\ref{sec:cls}--\ref{sec:exp_conv_analysis_hold}, we assumed that 
only the sensor-to-controller channel suffers from packet loss.
Here we briefly discuss the case where 
packets  simultaneously drop
in the sensor-to-controller and controller-to-actuator 
channels, which has been studied, e.g., in \cite{Liu2022}.
Note that in this case, 
the controller does not have to 
transmit notification packets to the actuator, because
the actuator can detect packet loss on its own.
As shown below, the same stability results as in
Theorems~\ref{thm:exp_conv} and \ref{thm:exp_conv2} 
can be obtained by slightly changing
the update rule of the zoom parameters.

\subsubsection{Zero compensation}
Let the controller and the actuator employ the zero compensation strategy.
Then the operators 
$B$ and $D$ in the plant \eqref{eq:plant} 
are replaced with $(1-\theta(k)) B$ and $(1-\theta(k))D$, respectively.
In the case without quantization, 
the closed-loop system with state space 
$Z \coloneqq X \times X_{\rm c}$ 
is given by $z(k+1) = \mathcal{A}_{\theta(k)}z(k)$
for $k \in \mathbb{N}_0$, where
\[
z(k) \coloneqq 
\begin{bmatrix}
	x(k) \\ x_{\rm c}(k)
\end{bmatrix},\hspace{7.8pt}
\mathcal{A}_{\theta} \coloneqq
\begin{bmatrix}
	A & (1-\theta)BR \\ (1-\theta)QC & P_\theta+(1-\theta)QDR
\end{bmatrix}
\]
for $\theta \in \{0,1 \}$.
The major difference from 
the case  in 
Sections~\ref{sec:cls}--\ref{sec:exp_conv_analysis_hold} is that 
the plant output
$y(k)$ depends on $\theta(k)$ as follows:
\begin{equation}
	\label{eq:y_simul_packet_loss}
	y(k) = \mathcal{C}_{\out,\theta(k)} z(k) + \mathcal{D}_{\theta(k)} (
	q_{\inc}(k) - u(k)
	),
\end{equation}
where  $\mathcal{C}_{\out,\theta}$
and $\mathcal{D}_{\theta}$ are defined by
\[
\mathcal{C}_{\out,\theta} \coloneqq 
\begin{bmatrix}
	C & (1-\theta) DR
\end{bmatrix},\quad 
\mathcal{D}_{\theta} \coloneqq 
(1-\theta) D
\]
for $\theta \in \{0,1 \}$.
However, since 
\begin{align}
	&\|\mathcal{C}_{\out,0}\|_{\mathcal{L}(Z,Y)}
	+ \Delta_{\inc} \|\mathcal{D}_{0}\|_{\mathcal{L}(U,Y)} \hspace{1pt}
	\|R\|_{\mathcal{L}(X_{\rm c},U)} \geq
	\|\mathcal{C}_{\out,1}\|_{\mathcal{L}(Z,Y)}
	+ \Delta_{\inc} \|\mathcal{D}_{1}\|_{\mathcal{L}(U,Y)} \hspace{1pt}
	\|R\|_{\mathcal{L}(X_{\rm c},U)},
	\label{eq:C_01_zero_comp}
\end{align}
it follows that $\mu_{\out}$ given in \eqref{eq:E_difference_eq}
can be used in this situation as well. Therefore, we obtain
the same results as in Theorem~\ref{thm:exp_conv},
by simply replacing the definition \eqref{eq:eta1_def} of $\eta_1$ with
$
\eta_1 \coloneqq M  \|\mathcal{A}_1\|_{\mathcal{L}(Z)},
$
where the constant $M\geq 1$ is as in Assumption~\ref{assump:power_stability}.

\subsubsection{Hold compensation}
Let the hold compensation strategy be implemented in 
the controller and the actuator. Let $\hat q_{\inc}(k)\in U$
be 
the last received $q_{\inc} $ that the actuator
keeps at time $k$, which is used as an auxiliary plant state 
as in
\eqref{eq:controller_hold_comp}.
The state space $Z$ of the closed-loop system is given by
$Z  \coloneqq X \times U \times X_{\rm c} \times Y$.
When quantization errors are zero, the closed-loop system can be written as
$z(k+1) = \mathcal{A}_{\theta(k)}z(k)$ for $k \in \mathbb{N}_0$,
where
\begin{align*}
	z(k) &\coloneqq
	\begin{bmatrix}
		x(k) \\ \hat{q}_{\inc}(k) \\ x_{\rm c}(k) \\ \hat{q}_{\out}(k)
	\end{bmatrix},\\
	\mathcal{A}_{\theta} &\coloneqq
	\begin{bmatrix}
		A & \theta B & (1-\theta)BR & 0 \\
		0 & \theta I & (1-\theta)R & 0 \\
		(1-\theta)QC & 0 & P+(1-\theta)QDR & \theta Q \\
		(1-\theta )C & 0 & (1-\theta)DR & \theta I
	\end{bmatrix}
\end{align*}
for $\theta \in \{0,1 \}$. 
As in the case of the zero compensation strategy, $y(k)$ is given by
\eqref{eq:y_simul_packet_loss}, where $\mathcal{C}_{\out,\theta}$
and $\mathcal{D}_{\theta}$ are replaced with 
\[
\mathcal{C}_{\out,\theta} \coloneqq
\begin{bmatrix}
	C & \theta D & (1-\theta)DR & 0
\end{bmatrix},\quad 
\mathcal{D}_{\theta} \coloneqq ( 1 - \theta ) D
\]
for $\theta \in \{0,1 \}$.
Since the inequality \eqref{eq:C_01_zero_comp} is not satisfied
in general
for the hold compensation strategy,
we need to modify
the definition of $\mu_{\out}$ as follows:
$
\mu_{\out}(k) \coloneqq \kappa \mu(k),
$
where the constant $\kappa >0$ is defined by
\[
\kappa \coloneqq \max_{\theta \in \{0,1 \}} (
\|\mathcal{C}_{\out,\theta}\|_{\mathcal{L}(Z,Y)}
+ \Delta_{\inc} \|\mathcal{D}_{\theta}\|_{\mathcal{L}(U,Y)} \hspace{1pt}
\|R\|_{\mathcal{L}(X_{\rm c},U)})  .
\]
There exists $M_{\rm h} \geq 1$ such that $\|\mathcal{A}_0\|_{\mathcal{L}(Z)} \leq M_{\rm h}  \rho^k$
for all $k \in \mathbb{N}$, where $\rho \in (0,1)$
is as in Assumption~\ref{assump:power_stability}.
The same results as in
Theorem~\ref{thm:exp_conv2} follow from 
replacing the definition \eqref{eq:eta1_def} of $\eta_1$
with 
$\eta_1 \coloneqq M_{\rm h} \|
\mathcal{A}_1
\|_{\mathcal{L}(Z)} $
and redefining $\eta_{\out}$ as above.

\section{Approximation of operator norms}
\label{sec:norm_approximation}
The quantizer design proposed in 
Section~\ref{sec:Quantization}
requires the norm computation of the operators
representing the system dynamics.
Here we present numerical methods to compute
the norms of those operators, which are based on 
finite-dimensional approximation.
In Section~\ref{sec:NE_by_FDA}, we  state two main results.
Sections~\ref{sec:Proof_of_app1} and 
\ref{sec:Proof_of_app2} are devoted to the proofs of these results,
respectively.

\subsection{Approximation by operators on finite-dimensional spaces}
\label{sec:NE_by_FDA}
Let $X,\,X_{\rm c},\,U$, and $Y$ be nontrivial
real or complex Banach spaces.
We consider operators $A \in \mathcal{L}(X)$, $B \in \mathcal{L}(U,X)$, and 
$C \in \mathcal{L}(X,Y)$. 
For $N \in \mathbb{N}$, 
let $\Pi_N \in \mathcal{L}(X)$ be a projection operator, 
i.e., satisfy $\Pi_N^2 = \Pi_N$, and let
$X_N \coloneqq \Pi_NX$ be finite dimensional.
We make the following assumptions throughout Section~\ref{sec:norm_approximation}.
\begin{assumption}
	\label{assump:A_Pi}
	The operator $A \in \mathcal{L}(X)$ and 
	the sequence $(\Pi_N)_{N \in \mathbb{N}}$
	of the projection operators in  $\mathcal{L}(X)$ satisfy
	\begin{enumerate}
		\renewcommand{\labelenumi}{\alph{enumi})}
		\item $\displaystyle \lim_{N \to \infty}\|(I-\Pi_N)x\|_X = 0$
		for all $x \in X$; \vspace{3pt}
		\item $\displaystyle \lim_{N \to \infty}\|(I-\Pi_N')x'\|_{X'} = 0$
		for all $x' \in X'$; and \vspace{3pt}
		\item
		$\displaystyle \lim_{N \to \infty}\!\| (I- \Pi_N)A\|_{\mathcal{L}(X)} =0 =\!\displaystyle \lim_{N \to \infty}\!\| A(I  - \Pi_N)\|_{\mathcal{L}(X)} $. \vspace{3pt}
	\end{enumerate}
\end{assumption}
\begin{assumption}
	\label{assump:finite_dimensional}
	The Banach spaces $U$ and $Y$  are finite-dimensional.
\end{assumption}

\begin{example}
	\label{ex:diagonalizable}
	A typical example of the pair 
	$A$ and $(\Pi_N)_{N \in \mathbb{N}}$
	satisfying
	Assumption~\ref{assump:A_Pi} can be constructed as follows
	using Riesz-spectral operators \cite[Section~3.2]{Curtain2020}
	or diagonalizable operators \cite[Section~2.6]{Tucsnak2009}:
	Let $X$ be a complex 
	Hilbert space  with inner product $\langle \cdot, \cdot \rangle$, and consider
	$A \in \mathcal{L}(X)$ of the form
	\[
	Ax = \sum_{n=1}^{\infty} \alpha_n \langle x, \psi_n \rangle \phi_n,\quad x \in X,
	\] 
	where $(\alpha_n)_{n \in \mathbb{N}}$ is a sequence of complex numbers,
	$(\phi_n)_{n \in \mathbb{N}}$ is a Riesz basis in $X$, and 
	$(\psi_n)_{n \in \mathbb{N}}$ is a biorthogonal sequence  to $(\phi_n)_{n \in \mathbb{N}}$.
	For $N \in \mathbb{N}$, let $\Pi_N \in \mathcal{L}(X)$ be
	a natural projection associated to the Riesz basis $(\phi_n)_{n \in \mathbb{N}}$, i.e.,
	\begin{equation}
		\label{eq:Pi_ex}
		\Pi_N x \coloneqq \sum_{n=1}^{N} \langle x, \psi_n \rangle \phi_n,\quad  x \in X.
	\end{equation}
	One easily verifies that
	Assumption~\ref{assump:A_Pi} is satisfied
	if
	$\lim_{n \to \infty}\alpha_n = 0$.
	\hspace*{\fill} $\triangle$ 
\end{example}

In the control problem studied in Section~\ref{sec:Quantization},
$X$, $U$, and $Y$  are
the state, input, and output spaces of the plant, respectively, 
and $X_{\rm c}$ is the state space of the controller.
As noted in Remark~\ref{rem:finite_dim_case},
if Assumption~\ref{assump:finite_dimensional} holds, then
we may consider the case where $X_{\rm c}$ is
finite dimensional,
under a reasonable condition on the plant. This allows us
to focus on finite-dimensional
approximation of $X$.

Let
$Z \coloneqq X \times X_{\rm c}$ and 
$Z_N \coloneqq   X_N \times X_{\rm c}$. Motivated by the ideal closed-loop operator 
$\mathcal{A}_{\rm id}$ given in \eqref{eq:A_id_def}, we
define the operator matrix
$\mathcal{A}\in \mathcal{L} (Z)$  by
\begin{equation}
	\label{eq:Acl_def}
	\mathcal{A}\coloneqq 
	\begin{bmatrix}
		A & BR \\
		QC & P
	\end{bmatrix}
\end{equation}
for
$P \in \mathcal{L}(X_{\rm c})$, 
$Q \in \mathcal{L}(U, X_{\rm c})$, and
$R \in  \mathcal{L}(X_{\rm c},Y)$.
As an approximation of $\mathcal{A}$, we consider the operator 
$\mathcal{A}_N \in \mathcal{L}(Z_N)$, $N \in \mathbb{N}$, defined by
\begin{equation}
	\label{eq:ANcl_def}
	\mathcal{A}_N \coloneqq
	\begin{bmatrix}
		A_N & B_NR \\
		QC_N & P
	\end{bmatrix},
\end{equation}
where 
\begin{equation}
	\label{eq:ABCN_def}
	A_N \coloneqq \Pi_N A|_{X_N},\quad 
	B_N \coloneqq \Pi_NB,\quad 
	C_N \coloneqq C|_{X_N}.
\end{equation}
For $N \in \mathbb{N}$,
we define the projection operator $\widetilde \Pi_N \in \mathcal{L}(Z)$ by
\begin{equation}
	\label{eq:tilde_Pi_def}
	\widetilde \Pi_N \coloneqq 
	\begin{bmatrix}
		\Pi_N & 0 \\
		0 & I
	\end{bmatrix}.
\end{equation}
Then $Z_N = \widetilde \Pi_N Z$.
Since we have from 
Assumption~\ref{assump:A_Pi}.a) that 
$\sup_{N \in \mathbb{N}} \|\Pi_Nx\| < \infty$ for all $x \in X$,
the uniform boundedness principle yields
$M_{\Pi}
\coloneqq \sup_{N \in \mathbb{N}} \|\Pi_N\|_{\mathcal{L}(X)} < \infty$. We also have that 
\begin{equation}
	\label{eq:tilde_Pi_bound}
	\|\widetilde \Pi_N\|_{\mathcal{L}(Z)} \leq M_{\Pi}
\end{equation} 
for all $N \in \mathbb{N}$.

Now we study the problem of approximating
operator norms.
This type of approximation has also been studied in the context of 
the Galerkin
finite element approximation for second-order elliptic boundary value problems; see, e.g., \cite[Theorem~2.9]{Kinoshita2020} and \cite[Theorem~3]{Kinoshita2023}.
For the approximation of the norm of $B$, we have  
$B - B_N = (I - \Pi_N)B$.
Since $B$ is of finite rank, it follows from Assumption~\ref{assump:A_Pi}.a) that
\begin{equation}
	\label{eq:B_PiN_conv}
	\lim_{N \to \infty}\|(I-\Pi_N)B \|_{\mathcal{L}(U,X)} = 0.
\end{equation}
However,
the operators
$\mathcal{A} - \mathcal{A}_N$ and $C - C_N$ are not well-defined. Indeed,
$\mathcal{A}$ is  an operator on $Z$, while
$\mathcal{A}_N$ is on $Z_N$.
Similarly, the domain $X$ of $C$ differs from the domain $X_N$ of $C_N$.
Despite the difference of the domains, we can prove that
the norms of $\mathcal{A}_N$ and $C_N$ converge to
those of $\mathcal{A}$ and $C$, respectively, by
using the technique developed for \cite[Theorem~2.9]{Kinoshita2020}.
The proof of the following  theorem is given in Section~\ref{sec:Proof_of_app1}.
\begin{theorem}
	\label{thm:norm_conv}
	Suppose that Assumptions~\ref{assump:A_Pi}  and \ref{assump:finite_dimensional} hold.
	For $A \in \mathcal{L}(X)$, $B \in \mathcal{L}(U,X)$, and
	$ C \in \mathcal{L}(X, Y)$, define their approximations $A_N$, $B_N$, and 
	$C_N$ by \eqref{eq:ABCN_def}, and
	define
	the operator $\mathcal{A}$ and its approximation
	$\mathcal{A}_N$ by \eqref{eq:Acl_def} and \eqref{eq:ANcl_def},
	respectively.
	Then the following statements hold:
	\begin{enumerate}
		\renewcommand{\labelenumi}{\alph{enumi})}
		\item
		$\displaystyle 	\|
		\mathcal{A}^k
		\|_{\mathcal{L}(Z)} = 
		\lim_{N \to \infty}\|
		\mathcal{A}_N^k
		\|_{\mathcal{L}(Z_N)}$ for all $k \in \mathbb{N}$.
		\vspace{3pt}
		\item
		$\displaystyle 	\|B\|_{\mathcal{L}(U, X)}  =  \lim_{N\to\infty}
		\|B_N\|_{\mathcal{L}(U, X_N)}$.
		\vspace{3pt}
		\item
		$\displaystyle \|C\|_{\mathcal{L}(X,Y)}  =  \lim_{N\to\infty}
		\|C_N\|_{\mathcal{L}(X_N,Y)} $.
		\vspace{3pt}
	\end{enumerate}
\end{theorem}	

\begin{remark}
	\label{rem:norm_computation_extension}
	Theorem~\ref{thm:norm_conv} is not directly applicable to
	the norm computation of the operators
	$\mathcal{B}_{\inc,0}$ and $\mathcal{C}_{\out}$ used in Sections~\ref{sec:exp_conv_analysis} and \ref{sec:exp_conv_analysis_hold}.
	However, we can approximate
	these norms in a similar way  by
	replacing $X$ with $X \times \mathbb{C}^n$ for some suitable $n \in \mathbb{N}$
	when $X_{\rm c}$ is finite dimensional.
	\hspace*{\fill} $\triangle$ 
\end{remark}	

\begin{remark}
	If Assumption~\ref{assump:A_Pi}.c) does not hold, then
	\begin{equation}
		\label{eq:PiN_conv}
		\lim_{N \to \infty}\|\Pi_N A|_{X_N}\|_{\mathcal{L}(X_N)} = 
		\|A\|_{\mathcal{L}(X)}
	\end{equation} 
	is not true in general.
	To see this, let $X \coloneqq \ell^2(\mathbb{N})$ and let $e_n \in \ell^2(\mathbb{N})$ be 
	the $n$-th unit vector in $\ell^2(\mathbb{N})$.
	Let  $A$ be 
	the right shift operator  on $\ell^2(\mathbb{N})$, i.e.,
	\[
	A (x_1,x_2,x_3,\dots ) \coloneqq
	(
	0 , x_1 ,  x_2,  \dots
	).
	\]
	Consider the 
	projections $\Pi_N$, $N \in \mathbb{N}$, in the form \eqref{eq:Pi_ex}, where
	the Riesz basis $(\phi_n)_{n \in \mathbb{N}}$ is defined by
	$\phi_{2\ell-1} \coloneqq e_{2\ell-1} + e_{2\ell}$ and
	$\phi_{2\ell} \coloneqq e_{2\ell}$ for $\ell \in \mathbb{N}$.
	Since
	\begin{align*}
		(I- \Pi_{2\ell+1} )Ae_{2\ell} &= -e_{2\ell+2}, \\
		A(I- \Pi_{2\ell-1} )e_{2\ell-1} &= -e_{2\ell+1}
	\end{align*}
	for all $\ell \in \mathbb{N}$, Assumption~\ref{assump:A_Pi}.c) is not satisfied.
	For all $\ell\in \mathbb{N}$, we obtain
	$\Pi_{2\ell+1} A  \phi_{2\ell}  = \phi_{2\ell+1}$,
	and hence
	\[
	\|\Pi_{2\ell+1} A |_{X_{2\ell+1}}\|_{\mathcal{L}(X_{2\ell+1})} \geq 2 > 1 = \|A\|_{\mathcal{L}(X)}.
	\]
	This implies that \eqref{eq:PiN_conv} is not true either.
	\hspace*{\fill} $\triangle$ 
\end{remark}

The convergence property 
given in statement~a) of Theorem~\ref{thm:norm_conv} is not uniform with respect to
$k \in \mathbb{N}$. 
Therefore, it may not be efficient for computing
the constants $M$ and $M_{\rm h}$ used in Theorems~\ref{thm:exp_conv} and \ref{thm:exp_conv2}.
When an upper bound on 
$\limsup_{N \to \infty} \|\Pi_N\|_{\mathcal{L}(X)}$ is known, we obtain an estimate for $\sup_{k \in \mathbb{N}_0}\|(\mathcal{A}/\rho_0)^k\|_{\mathcal{L}(Z)}$, where
$r(\mathcal{A}) < \rho_0$.
The proof can be found in Section~\ref{sec:Proof_of_app2}.
\begin{theorem}
	\label{thm:Acl_powered_bound}
	Suppose that Assumptions~\ref{assump:A_Pi} and
	\ref{assump:finite_dimensional} hold. Define
	the operator $\mathcal{A}$ and its approximation
	$\mathcal{A}_N$ by \eqref{eq:Acl_def} and \eqref{eq:ANcl_def},
	respectively. Let $M \geq 1$ and $\rho>0$ 
	satisfy 
	\[
	\|\mathcal{A}^k\| \leq M \rho^k
	\]
	for all $k \in \mathbb{N}_0$, and set
	\begin{align*}
		\widehat M(\rho_0) &\coloneqq
		\limsup_{N \to \infty} \sup_{k \in \mathbb{N}_0} \left\| \left(
		\frac{\mathcal{A}_N}{\rho_0}\right)^k
		\right\|_{\mathcal{L}(Z_N)} 
	\end{align*}
	for $\rho_0 > \rho$. Then
	\begin{equation}
		\label{eq:widehat_M}
		\widehat M(\rho_0) \leq M
	\end{equation}
	and 
	\begin{equation}
		\label{eq:Acl_powered_bound}
		\sup_{k \in \mathbb{N}_0}\left\| \left( \frac{
			\mathcal{A}}{\rho_0}\right)^k
		\right\|_{\mathcal{L}(Z)}  \leq 
		\widehat M(\rho_0) \limsup_{N \to \infty} \|\Pi_N\|_{\mathcal{L}(X)}
	\end{equation}
	for all $\rho_0 > \rho$.
\end{theorem}

\subsection{Proof of Theorem~\ref{thm:norm_conv}}
\label{sec:Proof_of_app1}
We start with an auxiliary
approximation result on
the powers of $\mathcal{A}$.
\begin{lemma}
	\label{lem:A_ANtilPi_power}
	Suppose that Assumptions~\ref{assump:A_Pi} and
	\ref{assump:finite_dimensional} hold. Then,
	the operator $\mathcal{A}$ and its approximation
	$\mathcal{A}_N$ defined by \eqref{eq:Acl_def} and \eqref{eq:ANcl_def},
	respectively, satisfy
	\begin{align}
		\lim_{N \to \infty}
		\|
		\mathcal{A}^k  -  \mathcal{A}_N^k\widetilde \Pi_N
		\|_{\mathcal{L}(Z)}&=
		0 \label{eq:APi_conv}
	\end{align}
	for all $k \in \mathbb{N}$, where the operator $\widetilde \Pi_N$ is as in \eqref{eq:tilde_Pi_def}.
\end{lemma}
\begin{proof}
	In the case $k = 1$, we have 
	\[
	\mathcal{A}  - \mathcal{A}_N \widetilde \Pi_N =
	\begin{bmatrix}
		A - \Pi_NA \Pi_N  & (I - \Pi_N) BR\\
		QC(I - \Pi_N) & 0
	\end{bmatrix}.
	\]
	By  the triangle inequality,
	\begin{align*}
		&\|A - \Pi_NA \Pi_N \|_{\mathcal{L}(X)}  \leq 
		\|(I - \Pi_N)A  \|_{\mathcal{L}(X)} + M_{\Pi} \|A(I -   \Pi_N) \|_{\mathcal{L}(X)},
	\end{align*}
	where $M_{\Pi} \coloneqq \sup_{N\in \mathbb{N}} \|\Pi_N \|_{\mathcal{L}(X)}$. 
	Recall that $M_{\Pi} < \infty $ by 
	Assumption~\ref{assump:A_Pi}.a) and the uniform boundedness
	principle.
	Using Assumption~\ref{assump:A_Pi}.c), we obtain
	\begin{equation}
		\label{eq:A_PiN_conv}
		\lim_{N \to \infty}\|A - \Pi_NA \Pi_N \|_{\mathcal{L}(X)}  =0.
	\end{equation}
	Since the dual operator $C' \in \mathcal{L}(Y',X')$ has finite rank, the following property similar to that of $B$ given in
	\eqref{eq:B_PiN_conv} is satisfied under Assumption~\ref{assump:A_Pi}.b):
	\begin{equation}
		\label{eq:C_PiN_conv}
		\lim_{N \to \infty} \hspace{-0.1pt}\|C(I-\Pi_N) \|_{\mathcal{L}(X, Y)}
		\hspace{-1pt}=\hspace{-1pt} \lim_{N \to \infty}\hspace{-0.1pt}
		\|(I-\Pi_N')C' \|_{\mathcal{L}(Y', X')}\hspace{-0.6pt}=\hspace{-0.6pt} 0.
	\end{equation}
	Hence, \eqref{eq:APi_conv} holds for $k = 1$.

	Assume that \eqref{eq:APi_conv} is true
	for some $k \in \mathbb{N}$. Since  $\mathcal{A}_N = \widetilde \Pi_N \mathcal{A}|_{Z_N}$,
	the triangle inequality together with \eqref{eq:tilde_Pi_bound}
	yields
	\begin{align*}
		\| \mathcal{A}^{k+1}  - 
		\mathcal{A}_N^{k+1} \widetilde \Pi_N \|_{\mathcal{L}(Z)} 
		&\leq
		\|(I - \widetilde \Pi_N) \mathcal{A} \|_{\mathcal{L}(Z)}
		\hspace{1pt} \|\mathcal{A}^k\|_{\mathcal{L}(Z)} + 
		M_{\Pi} \|\mathcal{A}\|_{\mathcal{L}(Z)} \hspace{1pt} 
		\| 
		\mathcal{A}^k - \mathcal{A}_N^k \widetilde \Pi_N \|_{\mathcal{L}(Z)}.
	\end{align*}
	By Assumption~\ref{assump:A_Pi}.c) and \eqref{eq:B_PiN_conv},
	$\|(I - \widetilde \Pi_N) \mathcal{A} \|_{\mathcal{L}(Z)} \to 0$ as 
	$N \to \infty$.
	This and the induction hypothesis show that 
	$\| \mathcal{A}^{k+1}  - 
	\mathcal{A}_N^{k+1} \widetilde \Pi_N \|_{\mathcal{L}(Z)} \to 0$
	as $N \to \infty$. Thus,
	\eqref{eq:APi_conv} holds 
	for all $k \in \mathbb{N}$ by induction.
\end{proof}

Now we are in a position to prove Theorem~\ref{thm:norm_conv}.
Although statement~c) on $C_N$ can be proved 
by slightly modifying the proof of \cite[Theorem~2.9]{Kinoshita2020},
we prove it here as a special case of statement~a) on $\mathcal{A}_N$.

\noindent\hspace{1em}{\itshape {\bf Proof of Theorem~\ref{thm:norm_conv}:} }
The proof is divided into three steps.

{\em Step 1: }
Statement~b) on $B_N$
immediately follows from \eqref{eq:B_PiN_conv}.
Moreover, we can prove statement~c) on $C_N$ from statement~a) on $\mathcal{A}_N$, considering 
the special case when $X_{\rm c} = Y$, $A=0$, $P=0$, $Q=I$, $R = 0$, and $k=1$.
It remains to prove statement~a)  on $\mathcal{A}_N$.

{\em Step 2: }	 
In this step,
we show that 
\begin{equation}
	\label{eq:Ak_upperbound_liminf}
	\liminf_{N \to \infty}\|\mathcal{A}_N^k\|_{\mathcal{L}(Z_N)}
	\geq 
	\|\mathcal{A}^k\|_{\mathcal{L}(Z)}
\end{equation}
for a fixed $k \in \mathbb{N}$.
To this end, it suffices to prove that 
the following inequality holds for all $z_0 \in Z$ with $\|z_0\|_Z = 1$:
\begin{equation}
	\label{eq:Ak_lowerbound}
	\liminf_{N \to \infty}\|\mathcal{A}_N^k\|_{\mathcal{L}(Z_N)}
	\geq \|\mathcal{A}^kz_0\|_{Z}.
\end{equation}

Let $z_0 \in Z$ satisfy $\|z_0\|_Z = 1$.
Since we have from Assumption~\ref{assump:A_Pi}.a) that 
\begin{equation}
	\label{eq:z_star_conv}
	\lim_{N \to \infty} \|\widetilde \Pi_Nz_0\|_Z = \|z_0 \|_Z = 1,
\end{equation}
there exists $N_1 \in \mathbb{N}$ such that $\|\widetilde \Pi_{N} z_0\|_Z >0$
for all $N \geq N_1$. 
Moreover,
\begin{equation}
	\label{eq:ANk_lowerbound}
	\|\mathcal{A}_N^k\|_{\mathcal{L}(Z_N)} = 
	\sup_{\substack{z \in Z_N \\ z \not=0}} \frac{ \|\mathcal{A}_N^k z\|_Z }{\|z\|_{Z}} \geq 
	\frac{ \| \mathcal{A}_N^k \widetilde \Pi_Nz_0\|_Z }{\| \widetilde \Pi_N z_0\|_{Z}} 
\end{equation}
for all $N \geq N_1$.
By Lemma~\ref{lem:A_ANtilPi_power}, we obtain
\begin{equation}
	\label{eq:AN_z_star_conv}
	\lim_{N \to \infty}
	\|\mathcal{A}_N^k \widetilde \Pi_Nz_0\|_Z = \| \mathcal{A}^k z_0 \|_Z.
\end{equation}
Combining \eqref{eq:z_star_conv}--\eqref{eq:AN_z_star_conv}, we obtain
the inequality
\eqref{eq:Ak_lowerbound}.

{\em Step 3: }
In the last step, we prove that
\begin{equation}
	\label{eq:limsupA_N}
	\|\mathcal{A}^k\|_{\mathcal{L}(Z)} \geq 
	\limsup_{N \to \infty}\|\mathcal{A}_N^k\|_{\mathcal{L}(Z_N)}
\end{equation}
for a fixed $k \in \mathbb{N}$. 
Together with \eqref{eq:Ak_upperbound_liminf}, this implies that statement~a) holds.

To get a contradiction, 
assume that the inequality \eqref{eq:limsupA_N} is not true.
Then there exist a constant $\varepsilon >0$ and 
a subsequence $(\mathcal{A}^k_{N_j})_{j \in \mathbb{N}}$ such that 
\begin{equation}
	\label{eq:A_lowerbound}
	\|\mathcal{A}^k\|_{\mathcal{L}(Z)} + 2\varepsilon \leq \|\mathcal{A}_{N_j}^k\|_{\mathcal{L}(Z_{N_j})} 
\end{equation}
for all $j \in \mathbb{N}$.

Take $\delta_1>0$ arbitrarily. By Assumption~\ref{assump:A_Pi}.c)  and \eqref{eq:B_PiN_conv},
there exists $j \in \mathbb{N}$ such that 
\[
\| (I - \Pi_{N_j})  A\|_{\mathcal{L}(X)} +
\| (I - \Pi_{N_j})  BR\|_{\mathcal{L}(X_{\rm c},X)} \leq \delta_1.
\]
This yields $\|
\mathcal{A}z - \mathcal{A}_{N_j} z
\|_Z \leq  \delta_1 \|z\|_Z$
for all $z \in Z_{N_j}$.
If there exist $\ell  \in \mathbb{N}$ and $\delta_\ell >0$ such that
\begin{equation}
	\label{eq:A_power_diff}
	\|
	\mathcal{A}^\ell z - \mathcal{A}_{N_j}^\ell  z
	\|_Z \leq \delta_\ell  \|z\|_Z
\end{equation}
for all $z \in Z_{N_j}$,
then
\begin{align*}
	\|
	\mathcal{A}^{\ell +1}z - \mathcal{A}_{N_j}^{\ell +1} z
	\|_Z & \leq 
	\| \mathcal{A}^\ell 
	(\mathcal{A} z - \mathcal{A}_{N_j} z) \|_Z
	+\|
	(\mathcal{A}^\ell - \mathcal{A}_{N_j}^\ell)  (\mathcal{A}_{N_j} z) 
	\|_Z  \\
	& \leq 
	( 
	\| \mathcal{A}^\ell \|_{\mathcal{L}(Z)} \delta_1 + 
	\|\mathcal{A}_{N_j}\|_{\mathcal{L}(Z_{N_j})} \delta_\ell 
	) \|z\|_Z
\end{align*}
for all $z \in Z_{N_j}$.
By \eqref{eq:tilde_Pi_bound}, we have
\begin{align*}
	\|\mathcal{A}_N \|_{\mathcal{L}(Z_N)} 
	\leq
	\| \widetilde \Pi_N \mathcal{A} \|_{\mathcal{L}(Z)}  
	\leq M_{\Pi} \|\mathcal{A}\|_{\mathcal{L}(Z)}
\end{align*}
for all $N \in \mathbb{N}$. From these observations, we define
\[
\delta_{\ell +1} \coloneqq 	\| \mathcal{A}^\ell \|_{\mathcal{L}(Z)} \delta_1 + 
M_{\Pi}\|\mathcal{A}\|_{\mathcal{L}(Z)}  \delta_\ell
\]
for $\ell  \in \mathbb{N}$.
Then the inequality \eqref{eq:A_power_diff} holds
for all $z \in Z_{N_j}$ and $\ell  \in \mathbb{N}$.
Moreover, 
for all $\ell  \in \mathbb{N}$,
there exists a constant $M_\ell  >0$, independent on $j$, such that 
$\delta_\ell  \leq M_\ell  \delta_1$.

Taking $\delta_1 = \varepsilon/M_k$ in the above argument, we see that
there exists $j \in \mathbb{N}$ such that 
\[
\|
\mathcal{A}^kz - \mathcal{A}_{N_j}^k z
\|_Z \leq \varepsilon \|z\|_Z
\]
for all $z \in Z_{N_j}$.
Therefore,
\begin{align*}
	\|\mathcal{A}_{N_j}^k\|_{\mathcal{L}(Z_{N_j})}
	\leq 
	\sup_{\substack{z \in Z_{N_j} \\ \|z\|_Z = 1}}
	\|\mathcal{A}^k z \|_Z + \varepsilon
	\leq \| \mathcal{A}^k\|_{\mathcal{L}(Z)} + \varepsilon.
\end{align*}
This contradicts \eqref{eq:A_lowerbound}.
Thus, the inequality \eqref{eq:limsupA_N} holds.
\hspace*{\fill} $\blacksquare$

\subsection{Proof of
	Theorem~\ref{thm:Acl_powered_bound}}
\label{sec:Proof_of_app2}
The following lemma is useful to estimate the norms of the powers of a bounded linear operator.
\begin{lemma}
	\label{lem:Phi_k_norm_bound}
	Let $W$ be a nontrivial real or complex Banach space, and
	let $\Phi \in \mathcal{L}(W)$.
	Let $M \geq 1$ and $\rho > 0$ satisfy
	\[
	\|\Phi^k\|_{\mathcal{L}(W)} \leq M \rho^k
	\]
	for all 
	$k \in \mathbb{N}_0$.
	Assume that
	a sequence $(\Phi_N)_{N \in \mathbb{N}}$ in $
	\mathcal{L}(W)$ satisfies
	\[
	\lim_{N \to \infty} \|\Phi - \Phi_N \|_{\mathcal{L}(W)} = 0.
	\]
	Then
	for all $\rho_0 > \rho$, the constant 
	$\widehat M_1(\rho_0)$ defined by
	\[
	\widehat M_1(\rho_0) \coloneqq
	\limsup_{N \to \infty} \sup_{k \in \mathbb{N}_0}\left\|
	\left( \frac{\Phi_N}{\rho_0} \right)^k
	\right\|_{\mathcal{L}(W)} 
	\]
	satisfies 
	\begin{equation}
		\label{eq:Phi_k_bound}
		\sup_{k \in \mathbb{N}_0} \left\| \left( \frac{\Phi}{\rho_0}\right)^k
		\right\|_{\mathcal{L}(W)}  \leq \widehat M_1(\rho_0) \leq M.
	\end{equation}
	\vspace{-5pt}
\end{lemma}
\begin{proof}
	We define  a new norm on the Banach space $W$
	as in the proof of Theorem~\ref{thm:exp_conv}:
	\begin{align*}
		|w|_W &\coloneqq \sup_{k \in \mathbb{N}_0} \|\rho^{-k}\Phi^k w\|_W,\quad 
		w \in W.
	\end{align*}
	Define 
	the  operator norm $|\cdot|_{\mathcal{L}(W)}$ by
	\[
	|S|_{\mathcal{L}(W)} \coloneqq \sup_{\substack{w \in W \\ w \not=0}}
	\frac{|Sw|_W}{|w|_W},\quad 
	S \in \mathcal{L}(W).
	\]
	Since the norm $|\cdot|_W$ has the properties given in \eqref{eq:new_norm_prop},
	we obtain
	$|\Phi|_{\mathcal{L}(W)} \leq \rho$ and
	\begin{equation}
		\label{eq:new_operator_norm_prop}
		\frac{\|S\|_{\mathcal{L}(W)}}{M} 
		\leq |S|_{\mathcal{L}(W)} \leq  M\|S\|_{\mathcal{L}(W)}
	\end{equation}
	for all $S \in \mathcal{L}(W)$.
	
	Let $\rho_0 > \rho$ be given.
	First, we show that $\widehat M_1(\rho_0) \leq M$.
	Define
	$\delta_N\coloneqq |\Phi - \Phi_N|_{\mathcal{L}(W)}$ for $N \in \mathbb{N}$.
	By assumption, 
	\begin{equation}
		\label{eq:alpha_0_conv}
		\lim_{N \to \infty }\delta_N = 0,
	\end{equation}
	and hence
	there exists $N_1 \in \mathbb{N}$ such that 
	$\rho +\delta_N \leq \rho_0$ for all $N\geq N_1$.
	This gives
	\begin{equation}
		\label{eq:Phi_N_bound}
		|\Phi_N|_{\mathcal{L}(W)} \leq 
		|\Phi|_{\mathcal{L}(W)} + |\Phi- \Phi_N|_{\mathcal{L}(W)} 
		\leq  \rho_0
	\end{equation}	
	for all  $N \geq N_1$.
	By \eqref{eq:Phi_N_bound},
	$|\Phi_N^k|_{\mathcal{L}(W)} \leq \rho_0^k$ for all $k \in \mathbb{N}_0$ and $N \geq N_1$. This and 
	\eqref{eq:new_operator_norm_prop} yield
	$
	\|\Phi_N^k\|_{\mathcal{L}(W)}  \leq M \rho_0^k
	$
	for all $k \in \mathbb{N}_0$ and $N \geq N_1$.
	Hence, the desired inequality $\widehat M_1(\rho_0) \leq M$ is obtained.
	
	Next, we prove that 
	\begin{equation}
		\label{eq:Phi_rho_bound}
		\sup_{k \in \mathbb{N}_0} \left\| \left( \frac{\Phi}{\rho_0}\right)^k
		\right\|_{\mathcal{L}(W)} \leq \widehat M_1(\rho_0).
	\end{equation}
	Let $N \geq N_1$, and define $\eta_N(1) \coloneqq \delta_N/\rho_0$.
	Then the inequality 
	\begin{equation}
		\label{eq:Phi_rho_power_bound}
		\left| 
		\frac{\Phi^k - \Phi^k_N}{\rho_0^k} \right|_{\mathcal{L}(W)} \leq \eta_N(k)
	\end{equation}
	is satisfied for $k = 1$.
	Define $\varsigma \coloneqq \rho/\rho_0 \in (0,1)$.
	We have
	$|(\Phi/\rho_0)^k|_{\mathcal{L}(W)} \leq \varsigma^k$.
	If
	$\eta_N(k) > 0$ satisfies \eqref{eq:Phi_rho_power_bound}
	for some $k \in \mathbb{N}$, then 
	\begin{align*}
		\left|\frac{\Phi^{k+1} - \Phi^{k+1} _N}{\rho_0^{k+1} } \right|_{\mathcal{L}(W)} &\leq 
		\left| \left( \frac{\Phi}{\rho_0} \right)^k 
		\frac{\Phi - \Phi_N}{\rho_0}\right|_{\mathcal{L}(W)}  +
		\left| \frac{\Phi^k - \Phi^k_N}{\rho_0^k}  \cdot
		\frac{\Phi_N}{\rho_0}  \right|_{\mathcal{L}(W)} \\
		&\leq \varsigma^k \eta_N(1)  + \eta_N(k).
	\end{align*}
	From this observation, we define
	\[
	\eta_N(k+1) \coloneqq \varsigma^k \eta_N(1)  + \eta_N(k) =
	\left(
	\frac{\rho}{\rho_0}
	\right)^k \frac{\delta_N}{\rho_0} + \eta_N(k)
	\]
	for $k \in \mathbb{N}$.
	Then
	the inequality \eqref{eq:Phi_rho_power_bound} holds for all $k \in \mathbb{N}$.
	Combining this with
	\eqref{eq:new_operator_norm_prop} and
	\eqref{eq:alpha_0_conv}, we obtain
	\begin{equation}
		\label{eq:Phi_diff_bound}
		\lim_{N \to \infty}
		\sup_{k\in \mathbb{N}_0}\left\| 
		\frac{\Phi^k - \Phi^k_N}{\rho_0^k} \right\|_{\mathcal{L}(W)} 
		\leq \lim_{N \to \infty} \frac{M\delta_N}{\rho_0 - \rho }= 0.
	\end{equation}
	From  \eqref{eq:Phi_diff_bound} and the triangle inequality
	\[
	\left\|\left( \frac{\Phi}{\rho_0}\right)^k\right\|_{\mathcal{L}(W)} \leq 
	\left\| 
	\frac{\Phi^k - \Phi^k_N}{\rho_0^k} \right\|_{\mathcal{L}(W)} + 
	\left\|\left( \frac{\Phi_N}{\rho_0}\right)^k\right\|_{\mathcal{L}(W)},
	\]
	we derive the desired inequality \eqref{eq:Phi_rho_bound}.
\end{proof}

We now turn to the proof of Theorem~\ref{thm:Acl_powered_bound}.

\noindent\hspace{1em}{\itshape {\bf Proof of Theorem~\ref{thm:Acl_powered_bound}:} }
Since  $\mathcal{A}_N = \widetilde \Pi_N \mathcal{A}|_{Z_N}$,
we obtain
\[
\mathcal{A}_N^k  =
(\mathcal{A}_N \widetilde \Pi_N)^k|_{Z_N}
\quad \text{and} \quad 
(\mathcal{A}_N \widetilde \Pi_N)^k  = 
\mathcal{A}_N^k \widetilde \Pi_N
\]
for all $k \in \mathbb{N}_0$. It follows that 
for all $k \in \mathbb{N}_0$,
\begin{equation}
	\label{eq:AN_lower_upper_bound}
	\| \mathcal{A}_N^k \|_{\mathcal{L}(Z_N)} \leq 
	\|(\mathcal{A}_N \widetilde \Pi_N)^k\|_{\mathcal{L}(Z)}
	\leq 
	\| \mathcal{A}_N^k \|_{\mathcal{L}(Z_N)}\hspace{1pt}\| \widetilde \Pi_N\|_{\mathcal{L}(Z)} .
\end{equation}

Let $\rho_0 >\rho$, and define
\[
\widehat{M_1}(\rho_0) \coloneqq \limsup_{N\to \infty} 
\sup_{k \in \mathbb{N}_0}
\left\|
\left(
\frac{
	\mathcal{A}_N
	\widetilde \Pi_N}{\rho_0}
\right)^k
\right\|_{\mathcal{L}(Z)}.
\]
Then,
it follows from \eqref{eq:AN_lower_upper_bound} that
\begin{equation}
	\label{eq:M_inf_inequalities}
	\widehat{M}(\rho_0) \leq \widehat{M_1}(\rho_0) \leq \widehat{M}(\rho_0) \limsup_{N \to \infty} \|\Pi_N\|_{\mathcal{L}(X)}.
\end{equation}
Lemma~\ref{lem:A_ANtilPi_power} with $k = 1$ and 
Lemma~\ref{lem:Phi_k_norm_bound}
give
\begin{align}
	\sup_{k \in \mathbb{N}_0}\left\| \left( \frac{
		\mathcal{A}}{\rho_0}\right)^k
	\right\|_{\mathcal{L}(Z)} 
	\leq \widehat{M_1}(\rho_0)\leq M. \label{eq:widehatM1_2}
\end{align}
Combining  \eqref{eq:M_inf_inequalities} and \eqref{eq:widehatM1_2},
we obtain the desired inequalities \eqref{eq:widehat_M} and 
\eqref{eq:Acl_powered_bound}.
\hspace*{\fill} $\blacksquare$

\section{Sampled-data Regular Linear Systems}
\label{sec:regular_sys}
In this section, we study the problem of stabilizing
sampled-data regular linear systems with
quantization and packet loss.
By discretizing the continuous-time plant  as in \cite[Section~3]{Logemann2013}, we can
apply the discrete-time results developed in Sections~\ref{sec:Quantization} and \ref{sec:norm_approximation}, which cover most aspects of the stabilization problem.
However,
two issues remain to be addressed: the inter-sample behavior of the  sampled-data
system and 
the numerical computation of the feedthrough matrix of the discretized plant.
First, we provide some preliminaries on regular linear systems and time discretization in Section~\ref{sec:prel}. Then
we analyze the inter-sample behavior in Section~\ref{sec:inter_sample}
and present a numerical method to compute 
the feedthrough matrix in 
Section~\ref{sec:comp_regular_case}.

\subsection{Preliminaries}
\label{sec:prel}
\subsubsection{Regular linear systems}
\label{sec:RLS}
We give preliminaries on regular linear systems; see the surveys \cite{Weiss2001, Tucsnak2014} and the book \cite{Staffans2005} for more details.
In Section~\ref{sec:regular_sys},
we consider
a regular linear system $\Sigma$ with state space $X$, input space $\mathbb{C}^m$, and 
output space $\mathbb{C}^p$.
Here $X$ is a nontrivial separable complex Hilbert space with 
inner product $\langle \cdot ,\cdot \rangle$.
Let $A$ be the semigroup generator of $\Sigma$, and 
let $(T(t))_{t \geq 0}$ be the $C_0$-semigroup on $X$ 
generated by $A$.
The control operator and the observation operator of $\Sigma$
are denote by $B$ and $C$, respectively. Let $D$ 
be the feedthrough matrix of $\Sigma$.

We denote by $X_1$ and $X_{-1}$ the interpolation and extrapolation spaces
associated to $(T(t))_{t \geq 0}$, respectively. Let $\lambda \in \varrho(A)$ be arbitrary.
The interpolation space $X_1$ is defined as $\dom (A)$ equipped with
the norm $\| \xi\|_{1} \coloneqq \|(\lambda I - A)\xi\|_X$.
The extrapolation space $X_{-1}$ is the completion of $X$ with
respect to the norm $\|\xi \|_{-1} \coloneqq \|(\lambda I - A)^{-1}\xi \|_X$.
The restrictions of $T(t)$
to $X_1$ form 
a $C_0$-semigroup on $X_1$, and
the generator of the restricted semigroup is the part of $A$ in $X_1$. Moreover,
$(T(t))_{t \geq 0}$ can be uniquely extended to a $C_0$-semigroup on $X_{-1}$, and
the generator of the extended semigroup is an extension of $A$ with 
domain $X$. For simplicity,
the same symbols are used for the restrictions and extensions of $(T(t))_{t \geq 0}$ and $A$.
Further  information on the interpolation and extrapolation spaces
can be found in
\cite[Section~3.6]{Staffans2005}
and  \cite[Section~2.10]{Tucsnak2009}.

The control operator $B$ and 
the observation operator $C$
belong  to  $\mathcal{L}(\mathbb{C}^m,X_{-1})$ and $ \mathcal{L}(X_1,\mathbb{C}^p)$, respectively. Furthermore, $B$ and $C$ are admissible
for $(T(t))_{t \geq 0}$; see, e.g.,
\cite[Chapter~10]{Staffans2005} and 
\cite[Chapter~4]{Tucsnak2009}
for the details of admissibility.
The $\Lambda$-extension $C_{\Lambda}$ of $C$ is defined by
\[
C_{\Lambda} \xi \coloneqq 
\lim_{\substack{\zeta \to \infty \\ \zeta \in \mathbb{R}}} C\zeta( \zeta I-A)^{-1}\xi
\]
with domain $\dom (C_{\Lambda})$ consisting of those $\xi \in X$ for which the limit exists.

Let $x$ and $y$
be the state and output, respectively, of the regular linear 
system $\Sigma$ with 
initial condition $x(0)=x^0 \in X$ and input $u \in L^2_{\rm loc}(\mathbb{R}_+,\mathbb{C}^m)$. 
Then
\begin{equation}
	\label{eq:solution_diff}
	x(t) = T(t)x^0 + \int^t_0 T(t-s) Bu(s) ds
\end{equation}
for all $t \geq 0$. In addition,
$x(t)  \in \dom (C_{\Lambda})$ for a.e. $t \geq 0$, and
\begin{equation}
	\label{eq:RLS}
	\left\{
	\begin{aligned}
		\dot x(t) &= Ax(t) + Bu(t);\quad x(0) = x^0,\\
		y(t) &= C_{\Lambda} x(t) + Du(t)
	\end{aligned}
	\right.
\end{equation}
for a.e. $t \geq 0$,
where the differential equation above is interpreted on $X_{-1}$.
The  input-output operator $G\colon 
L^2_{\rm loc}(\mathbb{R}_+,\mathbb{C}^m) \to 
L^2_{\rm loc}(\mathbb{R}_+,\mathbb{C}^p)$ of the regular linear system $\Sigma$ is given by
\begin{equation}
	\label{eq:IOmap_formula}
	(Gu)(t) = C_{\Lambda}
	\int^t_0 T(t-s) Bu(s) ds + Du(t)
\end{equation}
for all $u \in L^2_{\rm loc}(\mathbb{R}_+,\mathbb{C}^m)$ and a.e. $t\geq0$.
\subsubsection{Time discretization}
\label{sec:time_disc}
Let $\tau >0$ denote the sampling period.
The zero-order hold operator $\mathcal{H}_{\tau}\colon F(\mathbb{N}_0,\mathbb{C}^m)\to 
L^2_{\rm loc}(\mathbb{R}_+,\mathbb{C}^m)$ is defined by
\begin{equation}
	\label{eq:hold_def}
	(\mathcal{H}_{\tau}f)(k\tau + t) \coloneqq f(k),\quad t \in [0,\tau),\,k \in \mathbb{N}_0.
\end{equation}
The generalized sampling operator $\mathcal{S}_{\tau}\colon  L^2_{\rm loc}(\mathbb{R}_+,\mathbb{C}^p)
\to F(\mathbb{N}_0,\mathbb{C}^p)$ is defined by
\begin{equation}
	\label{eq:sampler_def}
	(\mathcal{S}_{\tau}g) (k) \coloneqq \int^\tau_0 w(t) g(k\tau+t) dt,\quad k \in \mathbb{N}_0,
\end{equation}
where  $w \in L^2(0,\tau)$ is a weighting function.

Now we construct the discretized plant consisting of  the regular linear system \eqref{eq:RLS},
the zero-order hold $\mathcal{H}_{\tau}$, and
the generalized sampler $\mathcal{S}_{\tau}$.
Define the operators $A_{\tau} \in \mathcal{L}(X)$, $B_\tau \in \mathcal{L} (\mathbb{C}^m,X)$, $C_{\tau} \in \mathcal{L}(X,\mathbb{C}^p)$, and
$D_{\tau} \in \mathcal{L}(\mathbb{C}^m,\mathbb{C}^p)$ by
\begin{align*}
	A_{\tau} &\coloneqq T(\tau), \\
	B_{\tau} v &\coloneqq \int^\tau_0 T(t) Bvdt,\quad v \in \mathbb{C}^m, \\
	C_{\tau}\xi &\coloneqq \int^\tau_0 w(t) C_{\Lambda} T(t) \xi dt,\quad  \xi\in X, \\
	D_{\tau} v &\coloneqq \int^\tau_0 w(t)(G(v \chi_{[0,\tau]}) )(t)dt,\quad v \in \mathbb{C}^m,
\end{align*}
where $\chi_{[0,\tau]}$ is the indicator function of the interval $[0,\tau]$.
Let $x^0\in X$ and 
$u = \mathcal{H}_{\tau}u_{\rm c}$, where $u_{\rm c} \in F(\mathbb{N}_0,\mathbb{C}^p)$.
The state $x$  of the regular linear system $\Sigma$, given in 
\eqref{eq:solution_diff}, satisfies
\begin{align}
	\label{eq:discretization_plant}
	\left\{
	\begin{aligned}
		x((k+1)\tau) &= A_\tau x(k\tau) + B_\tau u_{\rm c}(k), \\
		y_{\rm c}(k) &= C_\tau x(k\tau) + D_\tau u_{\rm c}(k)
	\end{aligned}
	\right.
\end{align}
for all $k \in \mathbb{N}_0$, where 
$y_{\rm c} \coloneqq \mathcal{S}_\tau y$; see \cite[Lemma~2]{Logemann2013}
for details. The system \eqref{eq:discretization_plant} is called 
the {\em discretized plant}.

\subsection{Inter-sample behavior of sampled-data systems with
	quantization and 
	packet loss}
\label{sec:inter_sample}
When there is no quantization or packet loss,
we consider the following discrete-time controller:
\begin{equation}
	\label{eq:digital_controller}
	\left\{
	\begin{aligned}
		x_{\rm c}(k+1)
		&= Px_{\rm c}(k) + Qy_{\rm c}(k),\\
		u_{\rm c}(k) &= R x_{\rm c}(k)
	\end{aligned}
	\right.
\end{equation}
for $k \in \mathbb{N}_0$
with initial state $x_{\rm c}(0) = x^0_{\rm c} \in X_{\rm c}$,
where $X_{\rm c}$ is a nontrivial complex Hilbert space, 
$P \in \mathcal{L}(X_{\rm c})$, 
$Q\in \mathcal{L}(\mathbb{C}^p, X_{\rm c})$, and 
$R \in \mathcal{L}(X_{\rm c},\mathbb{C}^m)$. 
The continuous-time plant \eqref{eq:RLS} is connected to
the discrete-time controller \eqref{eq:digital_controller} via
the following sampled-data feedback law:
\[
u = \mathcal{H}_{\tau}u_{\rm c}\quad \text{and} \quad  y_{\rm c} = \mathcal{S}_\tau y.
\]
A controller that uses the sampled output generated by $\mathcal{S}_\tau$  must be
strictly causal, i.e., must have no feedthrough term. This is because
the $k$-th sampled output is
generated from the plant output on the interval $[k\tau, (k+1)\tau)$.

Consider the setting where 
the closed-loop sampled-data system consisting of the discretized plant  \eqref{eq:discretization_plant}
and the controller  \eqref{eq:digital_controller} is subject to 
quantization and packet loss
as in Section~\ref{sec:zero_compensation}.
The next proposition shows that
the state $x$ of the plant 
exponentially converges to zero between
sampling instants if the convergence properties at sampling instants
obtained in Theorem~\ref{thm:exp_conv} are satisfied.
The same result holds for the other settings
studied in Sections~\ref{sec:exp_conv_analysis_hold}
and \ref{sec:sim_packet}.
\begin{proposition}
	\label{prop:dist_cont_stability}
	For the discretized plant \eqref{eq:discretization_plant} and the controller 
	\eqref{eq:digital_controller},
	let the closed-loop system with quantization 
	and packet loss be constructed as in Section~\ref{sec:zero_compensation},
	where the zoom parameters
	$\mu_{\inc}$ and $\mu_{\out}$ are defined by \eqref{eq:E_difference_eq}.
	If statements~a) and b) of Theorem~\ref{thm:exp_conv} hold, then
	there exist constants $\Omega_x \geq 1$ and $\omega>0$,
	independent of $E_0$, such that 
	\[
	\|x(t)\|_X \leq \Omega_x e^{-\omega t} E_0 
	\]
	for all $t \geq 0$ and $z^0 \in Z$ with $\|z^0\|_Z \leq E_0$.
\end{proposition}
\begin{proof}
	Let $k \in \mathbb{N}_0$.
	By \eqref{eq:solution_diff}, 
	\[
	x(k\tau+t) = T(t)x(k\tau) + \int^t_0 T(t-s)Bq_{\inc}(k)ds
	\]
	for all $t \in [0,\tau)$.
	By the admissibility of $B$,
	there exists $M_B>0$ such that 
	\[
	\left\| \int^t_0 T(t-s)Bq_{\inc}(k)ds
	\right\|_X \leq \sqrt{\tau} M_B \|q_{\inc}(k)\|_U
	\]
	for all $t \in [0,\tau)$.
	Since $\|u(k)\|_U \leq \mu_{\inc}(k)$, it follows from \eqref{eq:qe_u} that 	
	$\|q_{\inc}(k)\|_U \leq 
	(1+\Delta_{\inc})
	\mu_{\inc}(k)$. 	Define $
	M_T \coloneqq \sup_{0\leq t < \tau}  \|T(t)\|_{\mathcal{L}(X)}
	$. Then
	\begin{align*}
		\|x(k\tau+t)\|_X \leq M_T \|x(k\tau)\|_X + \sqrt{\tau}(1+\Delta_{\inc}) M_B
		\mu_{\inc}(k)
	\end{align*}
	for all $t \in [0,\tau)$.
	Thus,
	the convergence properties of $x(k\tau)$ and $\mu_{\inc}(k)$ yield
	the desired conclusion.
\end{proof}

\begin{remark}
	\label{rem:finite_dim_case_SD}
	Combining
	\cite[Proposition~3 and Theorem~9]{Logemann2013} 
	yields
	a necessary and sufficient condition
	for the existence of  a controller \eqref{eq:digital_controller}
	that satisfies Assumption~\ref{assump:power_stability} 
	for discretized plant \eqref{eq:discretization_plant}.
	Essentially, this condition is comprised of
	the controllability and observability of the unstable part of the 
	plant and the choice of a non-pathological sampling period;
	see \cite[statement (1) of Theorem~9]{Logemann2013} for
	further details.
	Moreover, it has been shown in
	\cite[Theorem~9]{Logemann2013} that 
	there exists a stabilizing controller if and only
	if there exists a {\em finite-dimensional} stabilizing controller,
	which implies that 
	the use of finite-dimensional controllers does not require
	additional conditions on the plant in our
	sampled-data stabilization problem.
	In the proof of this theorem,
	a design method for
	finite-dimensional stabilizing controllers
	has also been provided.
	\hspace*{\fill} $\triangle$ 
\end{remark}

\subsection{Series representation of feedthrough matrices}
\label{sec:comp_regular_case}
For the quantizer design,
it remains to compute the feedthrough matrix $D_{\tau}$.
Note that even when the feedthrough matrix $D$ of the regular linear  
system $\Sigma$
is zero, the
feedthrough matrix $D_{\tau}$ of the discretized plant \eqref{eq:discretization_plant}
is nonzero in general 
due to generalized sampling.
Although $D_{\tau}$ is a matrix, 
the input-output operator $G$
is used in the definition of $D_{\tau}$, which introduces
computational challenges.

Throughout Section~\ref{sec:comp_regular_case}, we assume that 
the semigroup generator $A$ of 
the regular linear system $\Sigma$ 
has the following spectral expansion as in Example~\ref{ex:diagonalizable}:
\begin{equation}
	\label{eq:A_rep}
	A\xi = \sum_{n=1}^{\infty} \lambda_n \langle \xi, \psi_n \rangle \phi_n,\quad \xi \in \dom(A),
\end{equation}
and
\begin{equation}
	\label{eq:A_domain_rep}
	\dom(A) = \left\{
	\xi \in X : \sum_{n=1}^{\infty} |\lambda_n|^2\hspace{1pt} |\langle \xi, \psi_n \rangle|^2 < \infty 
	\right\},
\end{equation}
where $(\lambda_n)_{n \in \mathbb{N}}$ is a sequence of complex numbers,
$(\phi_n)_{n \in \mathbb{N}}$ is a Riesz basis in $X$, 
$(\psi_n)_{n \in \mathbb{N}}$ is a biorthogonal sequence to $(\phi_n)_{n \in \mathbb{N}}$. 
Let 
the control operator $B \in \mathcal{L}(\mathbb{C}^m,X_{-1})$ and 
the observation operator $C \in \mathcal{L}(X_1,\mathbb{C}^p)$
of the regular linear 
system $\Sigma$ be given by
\begin{equation}
	\label{eq:BC_rep}
	B u = \sum_{\ell=1}^m b_\ell v_{\ell},~ v =
	\begin{bmatrix}
		v_1 \\ \vdots \\ v_m
	\end{bmatrix}
	\in \mathbb{C}^m;~~  Cx = 
	\begin{bmatrix}
		c_1 x \\ \vdots \\ c_p x
	\end{bmatrix}\hspace{-3pt},~ x \in X_1,
\end{equation}
where $b_1,\dots,b_n \in X_{-1}$ and $c_1,\dots,c_p \in \mathcal{L}(X_1,\mathbb{C})$.

Let $X_1^{\rm d}$ denote $\dom (A^*)$ endowed with the norm 
$\|\xi\|_1^{\rm d}\coloneqq \|(\overline{\lambda} I - A^*)\xi \|_X$ 
for $\overline{\lambda} \in \varrho(A^*)$.
Then $X_{-1}$ can be identified with 
the dual of $X_{1}^{\rm d}$ with respect to the pivot space $X$.
We denote by $\langle \xi_1,\xi_{-1}\rangle_{X_1^{\rm d},X_{-1}}$
the functional $\xi_{-1} \in X_{-1}$ applied to $\xi_1 \in X_1^{\rm d}$, so that
$\langle \cdot, \cdot \rangle_{X_1^{\rm d},X_{-1}}$
is linear in the first variable and antilinear in the second variable.
We also define 
\[
\langle \xi_{-1}, \xi_1\rangle_{X_{-1},X_1^{\rm d}} \coloneqq 
\overline{\langle \xi_1,\xi_{-1}\rangle_{X_1^{\rm d},X_{-1}}}
\]
for $\xi_{-1} \in X_{-1}$ and  $\xi_1 \in X_1^{\rm d}$.
If $\xi \in X$ and $\xi_1 \in X_{1}^{\rm d}$, then
$\langle \xi, \xi_1 \rangle_{X_{-1},X_1^{\rm d}} =
\langle \xi, \xi_1 \rangle$.
We refer to \cite[Sections~2.9 and 2.10]{Tucsnak2009} for details.

Define
\begin{equation}
	\label{eq:Bnl_Cjn_def}
	B_{n\ell} \coloneqq \langle
	b_{\ell},\psi_n 
	\rangle_{X_{-1},X_1^{\rm d}}\quad \text{and}\quad 
	C_{jn} \coloneqq c_j \phi_n
\end{equation}
for $n \in \mathbb{N}$, $\ell \in \{1,\dots, m \}$, and $j \in \{1,\dots,p  \}$.
Under certain assumptions on unboundedness of $B$ and $C$,
we obtain
a series 
representation of $D_{\tau}$, which is useful to numerically compute the matrix 
$D_{\tau}$.
\begin{theorem}
	\label{thm:discrete_feedthrough}
	Assume that the generating operators $A$, $B$, and $C$ of the regular linear 
	system $\Sigma$ satisfy
	the following conditions:
	\begin{enumerate}
		\renewcommand{\labelenumi}{\alph{enumi})}
		\item $A$ is given by \eqref{eq:A_rep} and \eqref{eq:A_domain_rep}, and  $(\lambda_n)_{n \in \mathbb{N}}$ satisfies
		$\liminf_{n \to \infty} |\lambda_n| >0$.
		\item
		$B$ and 
		$C$ are as in
		\eqref{eq:BC_rep}, and 
		there exist constants 
		$\beta,\gamma \geq 0$ with $\beta + \gamma \leq 1$ such that 
		\begin{equation}
			\sum_{n=1}^\infty 
			\frac{|B_{n\ell}|^2}{1+|\lambda_n|^{2\beta}} 
			< \infty \quad \text{and} \quad 
			\sum_{n=1}^\infty 
			\frac{|C_{jn}|^2}{1+|\lambda_n|^{2\gamma}} < \infty
			\label{eq:c_cond2}
		\end{equation}
		for all
		$\ell \in \{1,\dots, m \}$ and $j \in \{1,\dots,p  \}$, where
		$B_{n\ell}$ and $C_{jn}$ are defined by \eqref{eq:Bnl_Cjn_def}.
	\end{enumerate}
	Then the $(j,\ell)$-entry $(D_{\tau})_{j\ell}$ of 
	the feedthrough matrix $D_{\tau}$ defined as in Section~\ref{sec:inter_sample}
	has the following series representation for all
	$\ell \in \{1,\dots, m \}$ and $j \in \{1,\dots,p  \}$:
	\begin{equation}
		\label{eq:Dtau_jl_element}
		(D_{\tau})_{j\ell} = \sum_{n=1}^{\infty}
		C_{jn} B_{n\ell}
		\int^\tau_0 w(t)\varpi_n(t) dt + 
		D_{j\ell}\int^\tau_0 w(t) dt,
	\end{equation}
	where 
	\[
	\varpi_n(t) \coloneqq 
	\int^t_0 e^{s \lambda_n}ds,\quad 0\leq t \leq \tau,~n\in \mathbb{N}.
	\]
	\vspace{-5pt}
\end{theorem}

\begin{proof}
	Let	$\ell \in \{1,\dots, m \}$ and $j \in \{1,\dots,p  \}$.
	By
	\eqref{eq:IOmap_formula},
	the 
	$(j,\ell)$-entry $(D_{\tau})_{j\ell}$ of the feedthrough matrix $D_{\tau}$ can be written as
	\begin{equation}
		\label{eq:D_tau_rep}
		(D_{\tau})_{j\ell}  = 
		\int^\tau_0 w(t) H_{j\ell}(t) dt + 
		D_{j\ell}\int^\tau_0 w(t) dt,
	\end{equation}
	where
	\begin{equation}
		\label{eq:Hjl_def}
		H_{j\ell}(t) \coloneqq
		\lim_{\substack{\zeta  \to \infty \\ \zeta  \in \mathbb{R}}}
		c_j\zeta(\zeta I - A)^{-1}\int^t_0 T(s) b_{\ell} ds
	\end{equation}
	for a.e.~$t \geq 0$.
	
	First, we obtain a series representation
	\begin{equation}
		\label{eq:Tb}
		\int^t_0 T(s)b_{\ell} ds = \sum_{n=1}^{\infty} 
		B_{n\ell} \varpi_n(t)\phi_n
	\end{equation}
	for all $t \in [0,\tau]$.
	The basic fact on a Riesz basis given, e.g., in
	\cite[Lemma~3.2.4]{Curtain2020} and \cite[Proposition~2.5.2]{Tucsnak2009}
	leads to
	\begin{equation}
		\label{eq:g_ell_expansion}
		\int^t_0 T(s)b_{\ell} ds = \sum_{n=1}^{\infty} \left\langle 
		\int^t_0 T(s)b_{\ell} ds, \psi_n
		\right\rangle \phi_n.
	\end{equation}
	One can show that 
	\begin{align*}
		\left\langle
		\int^t_0 T(s) \xi_{-1} ds, \xi
		\right\rangle = 
		\left\langle
		\xi_{-1}, \int^t_0 T(s)^* \xi ds
		\right\rangle_{X_{-1},X_{1}^{\rm d}}
	\end{align*}
	for all $\xi_{-1} \in X_{-1}$ and $\xi \in X$. Hence, 
	\begin{align*}
		\left\langle
		\int^t_0 T(s) b_{\ell} ds, \psi_n
		\right\rangle =
		B_{n\ell} \varpi_n(t).
	\end{align*}
	Substituting this to \eqref{eq:g_ell_expansion}, we derive 
	the series representation \eqref{eq:Tb}.
	
	Next, we show that  
	\begin{equation}
		\label{eq:Hjell_series_rep}
		H_{j\ell}(t) = 
		\sum_{n=1}^\infty C_{jn}  B_{n\ell}
		\varpi_n(t) 
	\end{equation}
	for a.e.~$t \in [0,\tau]$,
	where $H_{j\ell}$ is defined by \eqref{eq:Hjl_def}.
	Let $\zeta \in \varrho(A) \cap \mathbb{R}_+$.
	Since $c_j(\zeta I-A)^{-1} \in \mathcal{L}(X,\mathbb{C})$, 
	it follows from \eqref{eq:Tb} that
	\[
	c_j (\zeta I - A)^{-1}  \int^t_0 T(s) b_{\ell} ds =
	\sum_{n=1}^{\infty} \frac{
		C_{jn}B_{n\ell}
		\varpi_n(t) }{\zeta - \lambda_n}
	\]
	for all $t \in [0,\tau]$.
	By assumption, there exists $n_1 \in \mathbb{N}$
	such that $\delta \coloneqq \inf_{n \geq n_1} |\lambda_n| >0$.
	For all $t \in [0,\tau]$ and $n \geq n_1$, we have
	\[
	|\varpi_n(t)| \leq \frac{1+\kappa}{|\lambda_n|},
	\quad \text{where $\kappa \coloneqq 
		\max\{1,\, e^{\tau \sup_{n \in \mathbb{N}} \re \lambda_n}\}$}.
	\]
	Since $\beta + \gamma \leq 1$,
	the Cauchy--Schwarz inequality gives
	\begin{align*}
		&\sum_{n=n_1}^\infty 
		\left| C_{jn}  B_{n\ell}
		\varpi_n(t) 
		\right|   \leq 
		\frac{1+\kappa }{\delta^{1-\beta-\gamma}} 
		\sqrt{\sum_{n=n_1}^\infty 
			\frac{
				| C_{jn}|^2}{|\lambda_n|^{2\gamma}}
			\sum_{n=n_1}^\infty 
			\frac{
				|  B_{n\ell} 
				|^2}{|\lambda_n|^{2\beta}} }
	\end{align*}
	for all $t \in [0,\tau]$.
	Hence,
	by the conditions given in \eqref{eq:c_cond2}, there exists $M_1 >0$
	such that 
	\begin{equation}
		\label{eq:c_phi_b_psi_bound}
		\sum_{n=1}^\infty 
		\left|C_{jn}  B_{n\ell}
		\varpi_n(t) 
		\right| \leq M_1
	\end{equation}
	for all $t \in [0,\tau]$.
	If $\zeta \geq 2\sup_{n \in \mathbb{N}} \re \lambda_n$, then
	$|\zeta /(\zeta  - \lambda_n)| \leq 2$ for all $n \in \mathbb{N}$. 
	The series representation \eqref{eq:Hjell_series_rep}
	follows from
	the dominated convergence theorem.

	For $t \in [0,\tau]$ and $N \in \mathbb{N}$, define
	\[
	H_{j\ell,N}(t) \coloneqq
	\sum_{n=1}^N  C_{jn}  B_{n\ell}
	\varpi_n(t).
	\]
	From \eqref{eq:c_phi_b_psi_bound},
	it follows that $|w(t)H_{j\ell,N}(t)| \leq M_1 |w(t)|$
	for all $t \in (0,\tau)$ and $N \in \mathbb{N}$. Therefore,
	the dominated convergence theorem yields
	\begin{align*}
		\int^\tau_0 w(t) H_{j\ell}(t)dt 
		&=\sum_{n=1}^\infty  C_{jn}  B_{n\ell}
		\int^{\tau}_0 w(t)\varpi_n(t) dt.
	\end{align*}
	Substituting this into \eqref{eq:D_tau_rep}, we obtain the desired series
	representation
	\eqref{eq:Dtau_jl_element}.
\end{proof}

\section{Simulation results}
\label{sec:example}
In this section, 
we consider the one-dimensional heat equation with boundary control and observation.
We present numerical results on the relationship between quantization coarseness and packet-loss duration.
Time responses of the quantized sampled-data system with packet loss are also shown. 
\subsection{Closed-loop system}
\subsubsection{Plant and controller}
Consider a metal bar of length one, which is insulated at the right end $\xi = 1$.
The temperature in the bar is controlled through heat flux
acting at the left end $\xi = 0$. We measure
the temperature at the right end $\xi = 1$.
Let $\zeta(\xi,t)$ denote the temperature of the bar at position $\xi \in [0,1]$
and time $t \geq 0$.
The model of the temperature distribution is given by
\begin{equation*}
	\left\{
	\begin{alignedat}{2}
		\frac{\partial \zeta}{\partial t}(\xi,t) &= \frac{\partial^2 \zeta}{\partial \xi^2}(\xi,t);\qquad 
		&&\hspace{-30pt}\zeta(\xi,0) = \zeta^0(\xi), \\
		\frac{\partial \zeta}{\partial \xi}(0,t) &= -u(t),
		\quad \frac{\partial \zeta}{\partial \xi}(1,t) = 0, \\
		y(t) &= \zeta(1,t); &&\hspace{-30pt} \xi \in [0,1],~t \geq 0.
	\end{alignedat}
	\right.
\end{equation*}
It is known that
this system is a regular linear system with state space $X = L^2(0,1)$, input space $\mathbb{C}$, and
output space $\mathbb{C}$. The generating operators $(A,B,C,D)$ 
are given by
$
Af = f''
$
with domain 
\[
\dom(A) = \{f \in W^{2,2}(0,1): f'(0) = 0,~ f'(1) = 0 \},
\]
$Bu = u\delta_{0}$ for $u \in \mathbb{C}$,
$Cf = f(1)$ for $f \in \dom (A)$, and 
$D= 0$,
where  $\delta_{0}$  is  the delta function supported at the point $\xi =0$.

The eigenvalues of $A$ are 
$\lambda_n  \coloneqq -(n-1)^2\pi^2$ for $n \in \mathbb{N}$.
The eigenvector $\phi_n \in L^2(0,1)$ of $A$ with the
eigenvalue $\lambda_n$ is given by
$\phi_1(\xi)\equiv 1$ and
$
\phi_n(\xi) = \sqrt{2} \cos((n-1)\pi \xi)
$
for $\xi \in [0,1]$ and $n \geq 2$.
The sequence $(\phi_n)_{n \in \mathbb{N}}$ forms an orthonormal basis
in $X$. Hence, Assumptions~\ref{assump:A_Pi}
and \ref{assump:finite_dimensional} are satisfied, where
the projection operator $\Pi_N$ is as in \eqref{eq:Pi_ex}
for each $N \in \mathbb{N}$.
We also have $ \|
\Pi_N
\|_{\mathcal{L}(X)} = 1$ for all $N \in \mathbb{N}$.

Let the sampling period $\tau$ be $\tau = 0.1$, and 
let the weighting function $w$ for the sampler \eqref{eq:sampler_def} be 
$w(t) \equiv 1/\tau$.
Following the construction method of stabilizing controllers
given in the proof of \cite[Theorem~9]{Logemann2013}, we obtain
the discrete-time controller \eqref{eq:digital_controller} with $(P,Q,R) = (0.445,0.3,-3)$. Then 
the ideal closed-loop operator consisting of the discretized plant  \eqref{eq:discretization_plant}
and the controller  \eqref{eq:digital_controller},
\[
\mathcal{A}_{\rm id} \coloneqq
\begin{bmatrix}
	A_{\tau} & B_{\tau}R  \\ QC_{\tau} & P+QD_{\tau}R
\end{bmatrix},
\] 
satisfies $\sup_{k \in \mathbb{N}}\|(\mathcal{A}_{\rm id}/\rho)^k\|_{\mathcal{L}(X \times \mathbb{C})} < \infty$ with $\rho = 0.908$. Therefore,
Assumption~\ref{assump:power_stability} is satisfied.
We set $P_1 = P$ for the controller \eqref{eq:controller_zero} equipped with 
the zero compensation strategy.

\subsubsection{Quantizers}
For $L \in \mathbb{N}$, let the function 
$\mathsf{Q}_L\colon \mathbb{R} \to \mathbb{R}$ 
act on the interval $[-1,1]$ as uniform quantization with range $[-1,1]$ and step-size $2/L$.
The parameter $L$ means the number of quantization levels.
The quantization function $
\mathsf{Q}_{\inc} $ for the plant input 
is defined as $\mathsf{Q}_L$ with $L = L_{\inc}\in \mathbb{N}$.
Similarly, the quantization function $\mathsf{Q}_{\out}$ for the plant output
is defined as $\mathsf{Q}_L$ with $L = L_{\out}\in \mathbb{N}$.
Then the error bounds $\Delta_{\inc}$ and $\Delta_{\out}$ are given by
$\Delta_{\inc} = 1/L_{\inc}$ and $\Delta_{\out} = 1/L_{\out}$.

Some operator norms are used in
the proposed quantizer design; 
see \eqref{eq:E_difference_eq}--\eqref{eq:ex_conv_cond}.
To compute them numerically, we employ Theorems~\ref{thm:norm_conv},
\ref{thm:Acl_powered_bound}, and \ref{thm:discrete_feedthrough}.
For example, here we consider the operator $\mathcal{A}$ 
on $X \times \mathbb{C}$ defined by
$
\mathcal{A} \coloneqq  \mathcal{A}_{\rm id}
$
and its finite-dimensional approximation $\mathcal{A}_N$ on 
$X_N \times \mathbb{C}$ defined by
\[
\mathcal{A}_N \coloneqq
\begin{bmatrix}
	\Pi_N A_{\tau}|_{X_N} & \Pi_N B_{\tau}R  \\ QC_{\tau}|_{X_N} & P+QD_{\tau}R
\end{bmatrix}.
\]
Let $\rho_0 > 0.908$ and
\[
M_N(\rho_0) \coloneqq \sup_{k \in \mathbb{N}}  \left\| \left(
\frac{\mathcal{A}_N}{\rho_0}\right)^k
\right\|_{\mathcal{L}(X_N\times \mathbb{C})}.
\]
Then Theorem~\ref{thm:Acl_powered_bound} shows that
\[
\sup_{k \in \mathbb{N}} \left\| \left(
\frac{\mathcal{A}}{\rho_0}\right)^k
\right\|_{\mathcal{L}(X\times \mathbb{C})}  \leq  \limsup_{N \to \infty} 
M_N(\rho_0).
\]
Now 
consider $X \times \mathbb{C}$ to be equipped with the norm 
\[
\left\|
\begin{bmatrix}
	x \\ x_{\rm c}
\end{bmatrix}	
\right\|_{X\times \mathbb{C}} \coloneqq (\|x\|_X^2 + |x_{\rm c}|^2)^{1/2}.
\]
Figure~\ref{fig:Mn_conv} illustrates the constant $M_N(\rho_0)$ versus 
the approximation order $N$ in the case 
$\rho_0= 0.91$.
Based on the numerical computations,
$\limsup_{N\to \infty} M_N(\rho_0)$ is found to be approximately $1.3770$.
For the simulations in
Sections~\ref{sec:relation} and \ref{sec:time_resp} below,
the constants $M$ and $\rho$ 
used in Assumption~\ref{assump:power_stability} 
are set to $1.3770$ and $0.91$, respectively.
Similar 
numerical results are obtained for the other norms such as 
$\|\mathcal{B}_{\inc,0}\|_{\mathcal{L}(\mathbb{C}, X\times \mathbb{C})}$.

\begin{figure}[tb]
	\centering
	\includegraphics[width = 10cm,clip]{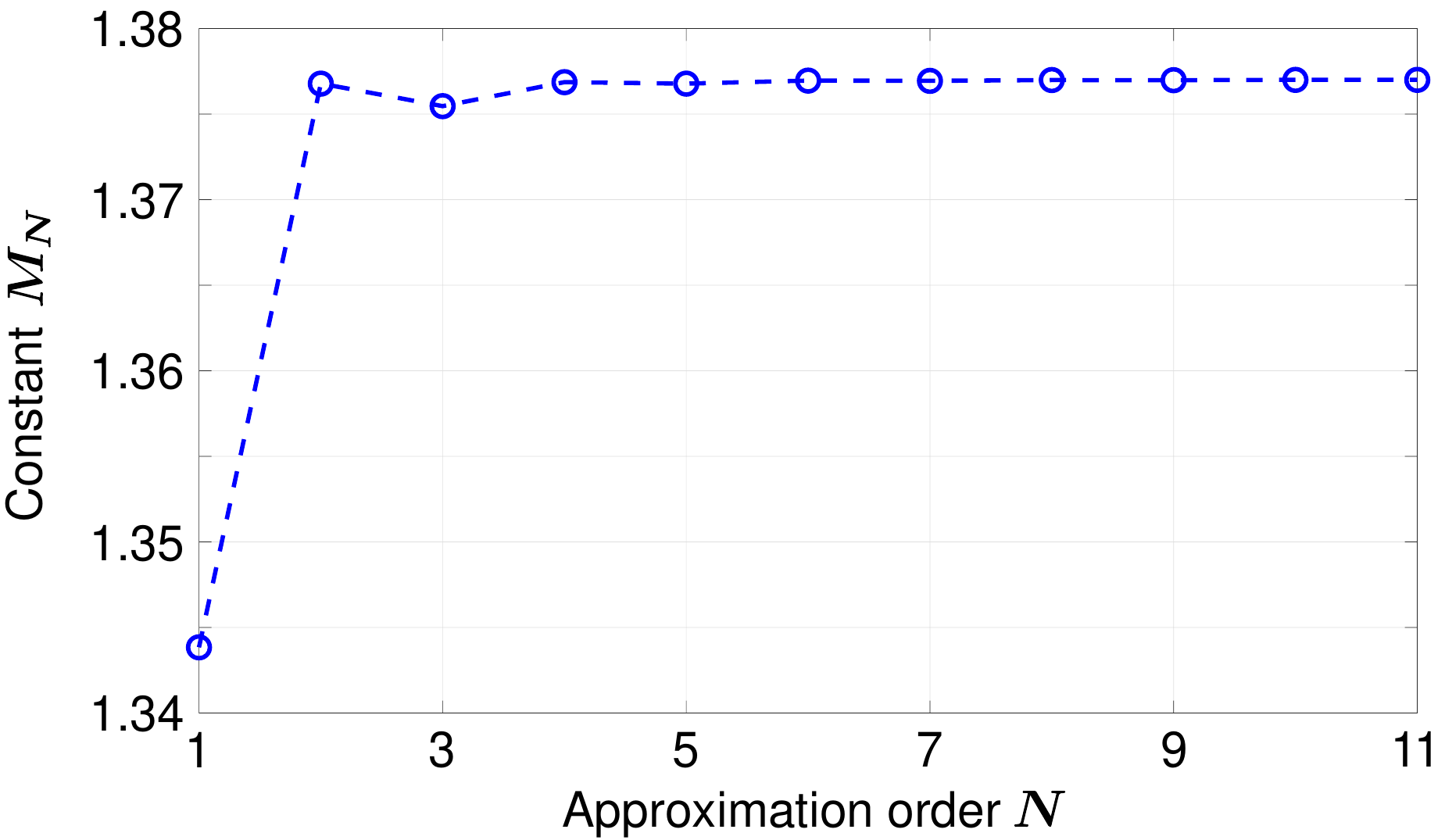}
	\caption{Constant $M_{N}(\rho_0)$ versus approximation order $N$ 
		in the case $\rho_0 = 0.91$.}
	\label{fig:Mn_conv}
\end{figure}

\subsection{Relationship between quantization coarseness 
	and packet-loss duration}
\label{sec:relation}
From Theorems~\ref{thm:exp_conv} and \ref{thm:exp_conv2}, 
we obtain relationships 
between the numbers $L_{\inc},L_{\out}$  of quantization levels and
the duration bound $\nu$ of packet loss  in the cases of
zero compensation and hold compensation,
respectively.
Figures~\ref{fig:zero_uystep_nu} and \ref{fig:hold_uystep_nu} show
the color map of 
the upper bound $
\log (1/\eta_0  )/ \log(\eta_1 / \eta_0)
$ of  $\nu$  given in Theorems~\ref{thm:exp_conv}
and \ref{thm:exp_conv2}, respectively,
as a function of  $L_{\inc}$ and $L_{\out}$.
The black solid lines are the contours of $
\log (1/\eta_0  )/ \log(\eta_1 / \eta_0)
$, which are plotted
at $0.06, 0.10, 0.14, 0.18$ in Figure~\ref{fig:zero_uystep_nu} and 
at $0.02, 0.04, 0.06,0.08 $ in Figure~\ref{fig:hold_uystep_nu}.
From these figures, we see that 
the zero compensation strategy tolerates more 
frequent packet loss and coarser quantization
than the hold compensation strategy.
In addition, compared with the zero compensation strategy,
the hold compensation strategy places more importance on
output quantization than on input quantization.

\begin{figure}
	\centering
	\subcaptionbox{ 
		Zero compensation.
		\label{fig:zero_uystep_nu}}
	{\includegraphics[width = 8cm,clip]{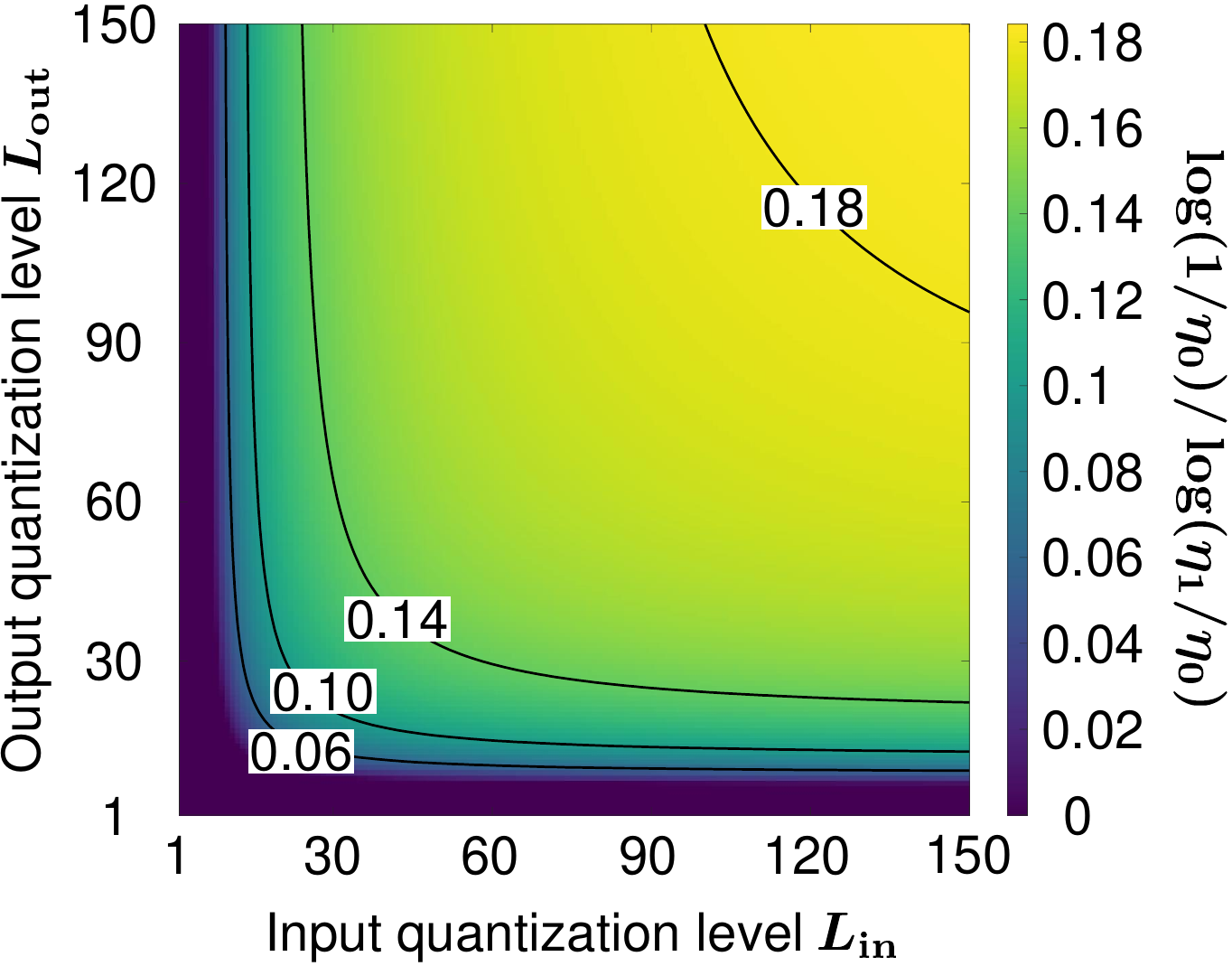}} \hspace{10pt}
	\subcaptionbox{Hold compensation.
		\label{fig:hold_uystep_nu}} 
	{\includegraphics[width = 8cm,clip]{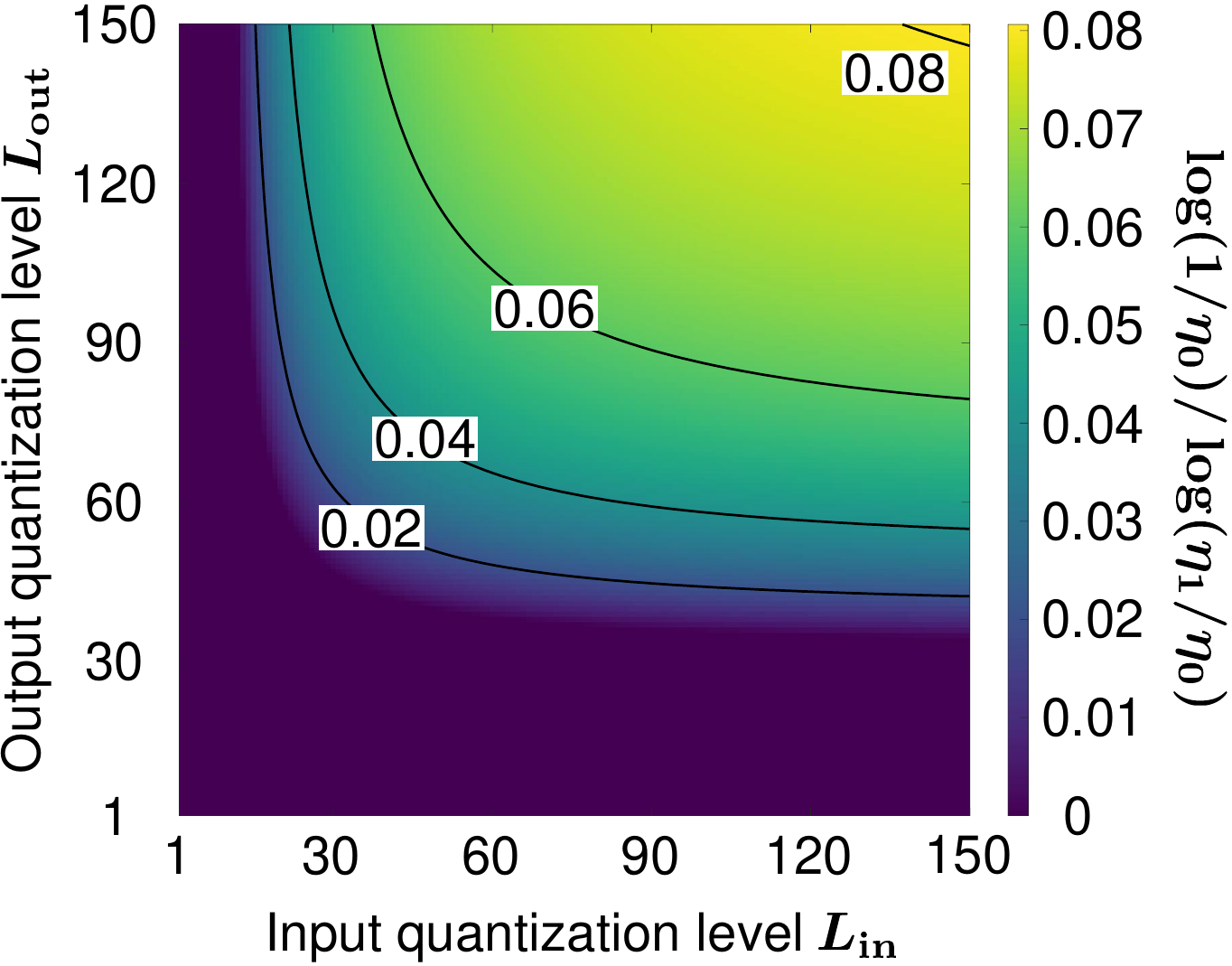}}
	\caption{Relationships between 
		numbers $L_{\inc}, L_{\out}$  of quantization levels and duration bound $\nu$
		of packet loss.
		\label{fig:uystep_nu}}
\end{figure}

\subsection{Time responses}
\label{sec:time_resp}
Let
the initial states $x^0 \in X$ and $x_{\rm c}^0 \in \mathbb{C}$ 
of the plant and the controller be given by
$x^0(\xi) \equiv 1$ and $x_{\rm c}^0 = 0$.
We set the initial value $\hat q_{\rm \out}^{\,0} $ 
of the auxiliary controller state to $0$ for the hold compensation strategy.
The initial state bound $E_0$ is assumed to be given by
$E_0 = 1$. The numbers $L_{\inc},L_{\out}$ of quantization levels are set to
$150$.	
Here we consider
the situation where packet loss occurs only
in the sensor-to-controller channel.

First, we show simulation results on the
zero compensation strategy.
Figure~\ref{fig:zero_time_response} depicts the time responses of
the controller input $q_{\rm out}$, 
the sampled plant output $y_{\rm c} = \mathcal{S}_{\tau}y$, 
the plant input $\mathcal{H}_{\tau}q_{\rm in}$, and
the controller output $u_{\rm c}  = Rx_{\rm c}$ in
the case $(\Xi,\nu) = (1,0.175)$, where
$q_{\rm out}(k)$ is generated from $y_{\rm c}(k)$ by 
the output quantizer $Q_{\out}$ and
$(\mathcal{H}_{\tau}q_{\rm in})(k\tau)$ is generated  from $u_{\rm c}(k)$ by
the input quantizer $Q_{\inc}$.
In Figure~\ref{fig:zero_time_response}, the packet containing $q_{\rm out}$ is lost
on the intervals that are colored in gray.

We observe from Figure~\ref{fig:zero_time_response} that 
the sampled plant output and the plant input converge to zero.
However, small oscillations occur
due to coarse quantization and the even number of quantization levels.
We also see from the time responses on the interval $[4,6]$ that 
the quantization errors become large after a packet transmission fails.
This is because
the zoom parameters of the quantizers
increase when packet loss is detected. 
In order to avoid quantizer saturation,
the zoom parameters decrease at a slower rate 
than the closed-loop state in the absence of packet loss
and increase at a higher rate in the presence of packet loss.
Therefore, the quantization errors are large 
compared with the outputs of the plant and the controller for $t \geq 2$.

Next, we present simulation results on the
hold compensation strategy. 
Figure~\ref{fig:hold_time_response} shows the time responses of the same signals
as Figure~\ref{fig:zero_time_response}, where
$(\Xi,\nu) = (1,0.072)$.
Gray-colored intervals indicate periods in which packet loss occurs.
As in the case of zero compensation,
we see that
the sampled plant output and the plant input converge to zero with small oscillations.

By comparing the time responses of the plant input on the interval $[0,1]$ in
Figures~\ref{fig:zero_input} and \ref{fig:hold_input}, 
we find the difference of the input behavior after packet loss between
the zero compensation strategy and the hold compensation strategy.
Moreover,
the time responses on the interval $[4,6]$ show that 
when the packet is dropped,
the quantization errors in the
hold compensation case become 
larger  than those in the zero compensation case.
This is because the rate of increase of the zoom parameters 
is given by $\eta_1 =1.4626$ in the zero compensation case and 
by $\eta_1 = 2.1515$ in the hold compensation case.

\begin{figure}
	\centering
	\subcaptionbox{ 
		Controller input $q_{\rm out}$ and sampled plant output $y_{\rm  c}$.
		\label{fig:zero_output}}
	{\includegraphics[width = 10cm,clip]{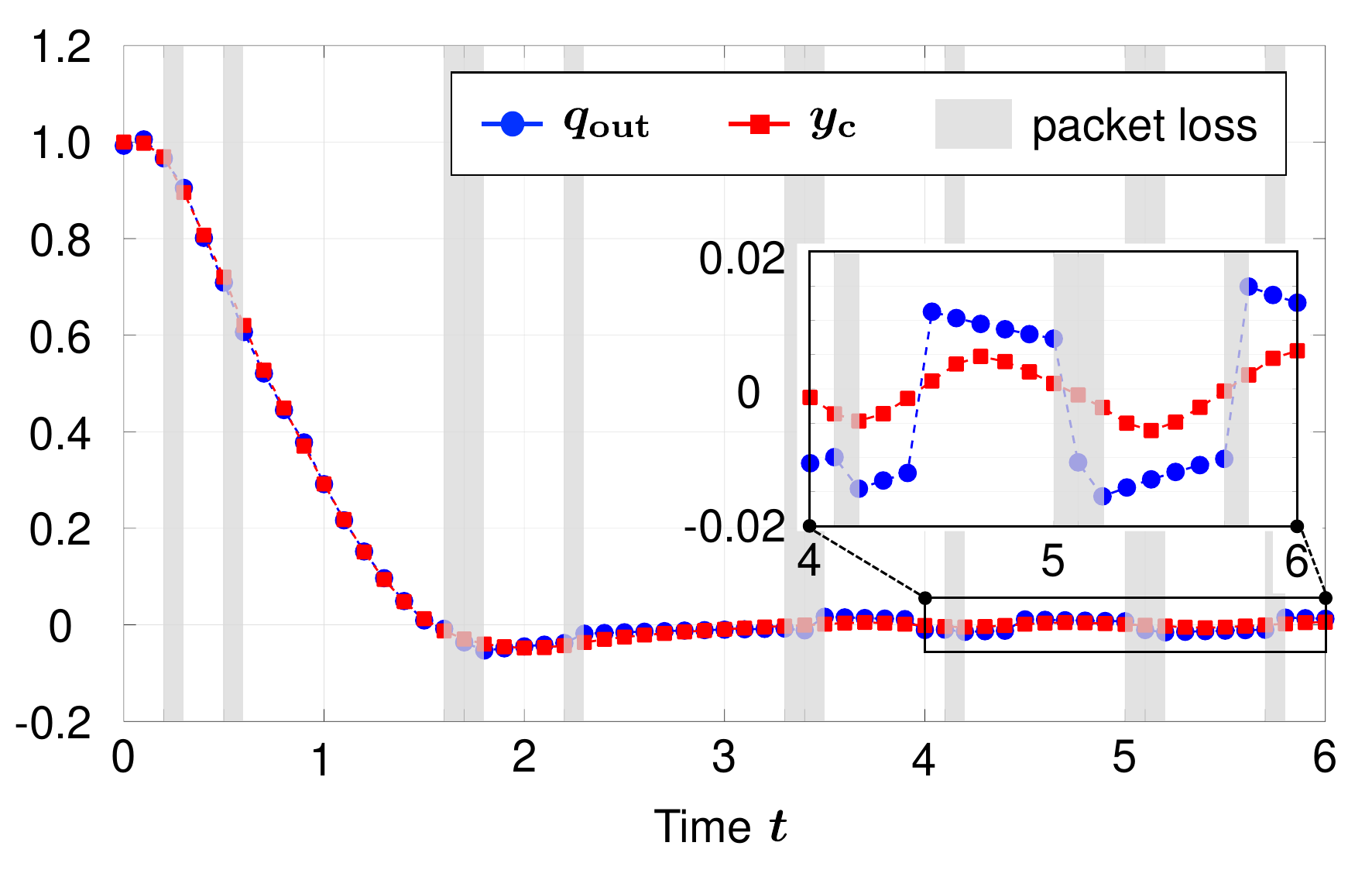}} 
	\subcaptionbox{Plant input $\mathcal{H}_{\tau}q_{\inc}$ and 
		controller output 
		$u_{\rm c}$.
		\label{fig:zero_input}}
	{\includegraphics[width = 10cm,clip]{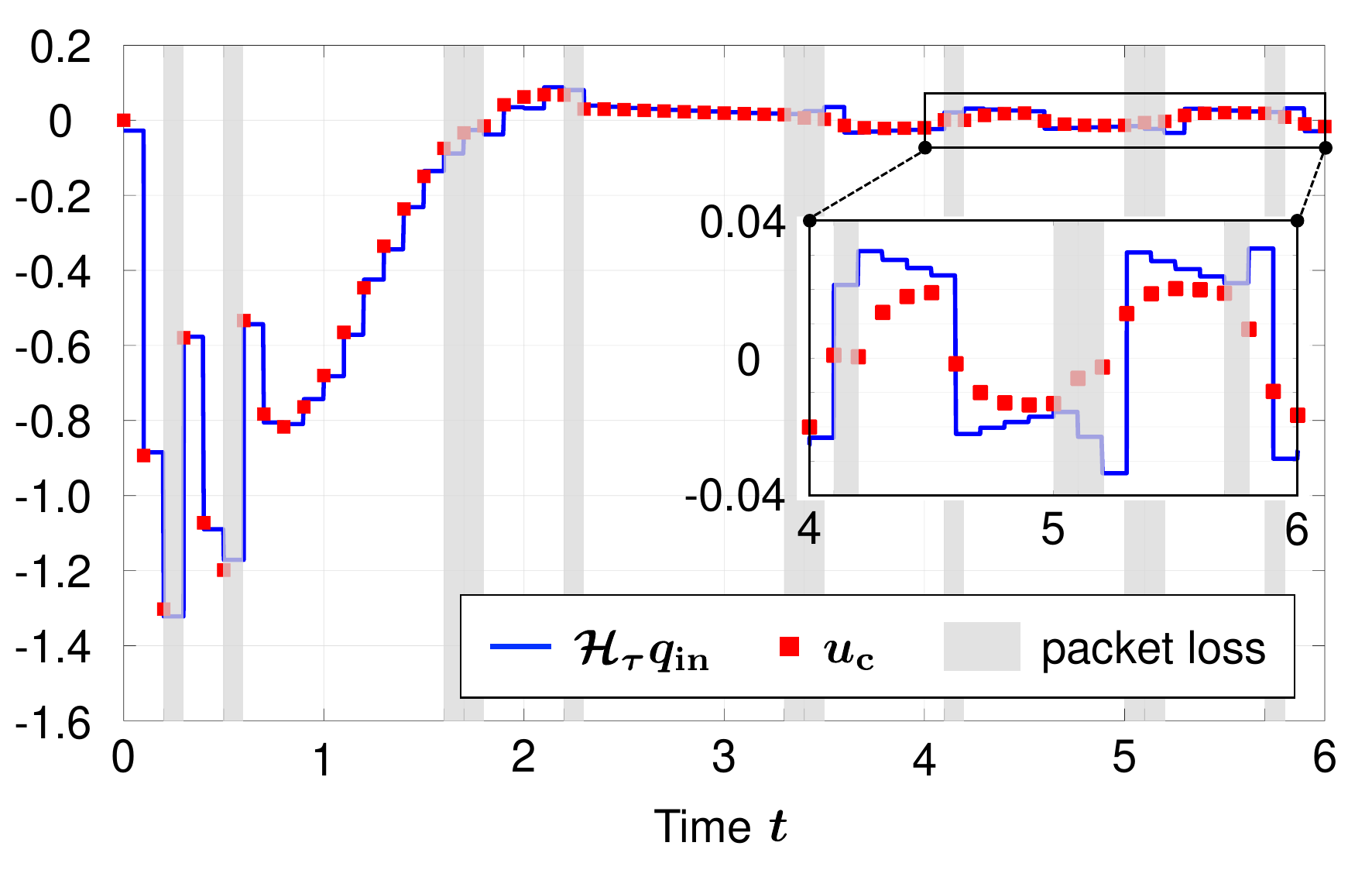}}
	\caption{Time responses under zero compensation.
		\label{fig:zero_time_response}}
\end{figure}

\begin{figure}
	\centering
	\subcaptionbox{ 
		Controller input $q_{\rm out}$ and sampled plant output $y_{\rm  c}$.
		\label{fig:hold_state}}
	{\includegraphics[width = 10cm,clip]{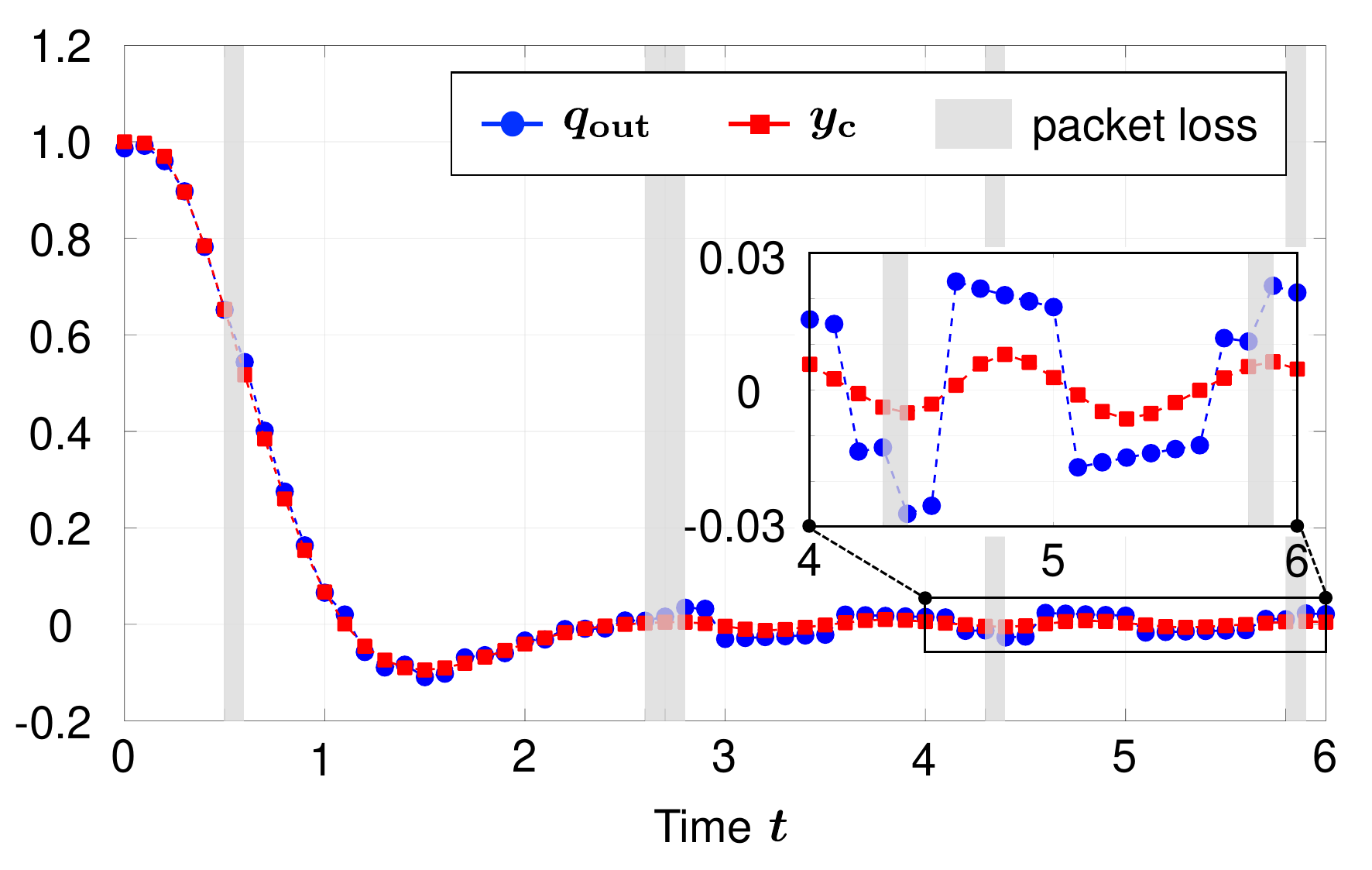}} 
	\subcaptionbox{Plant input $\mathcal{H}_{\tau}q_{\inc}$ and 
		controller output 
		$u_{\rm c}$.
		\label{fig:hold_input}}
	{\includegraphics[width = 10cm,clip]{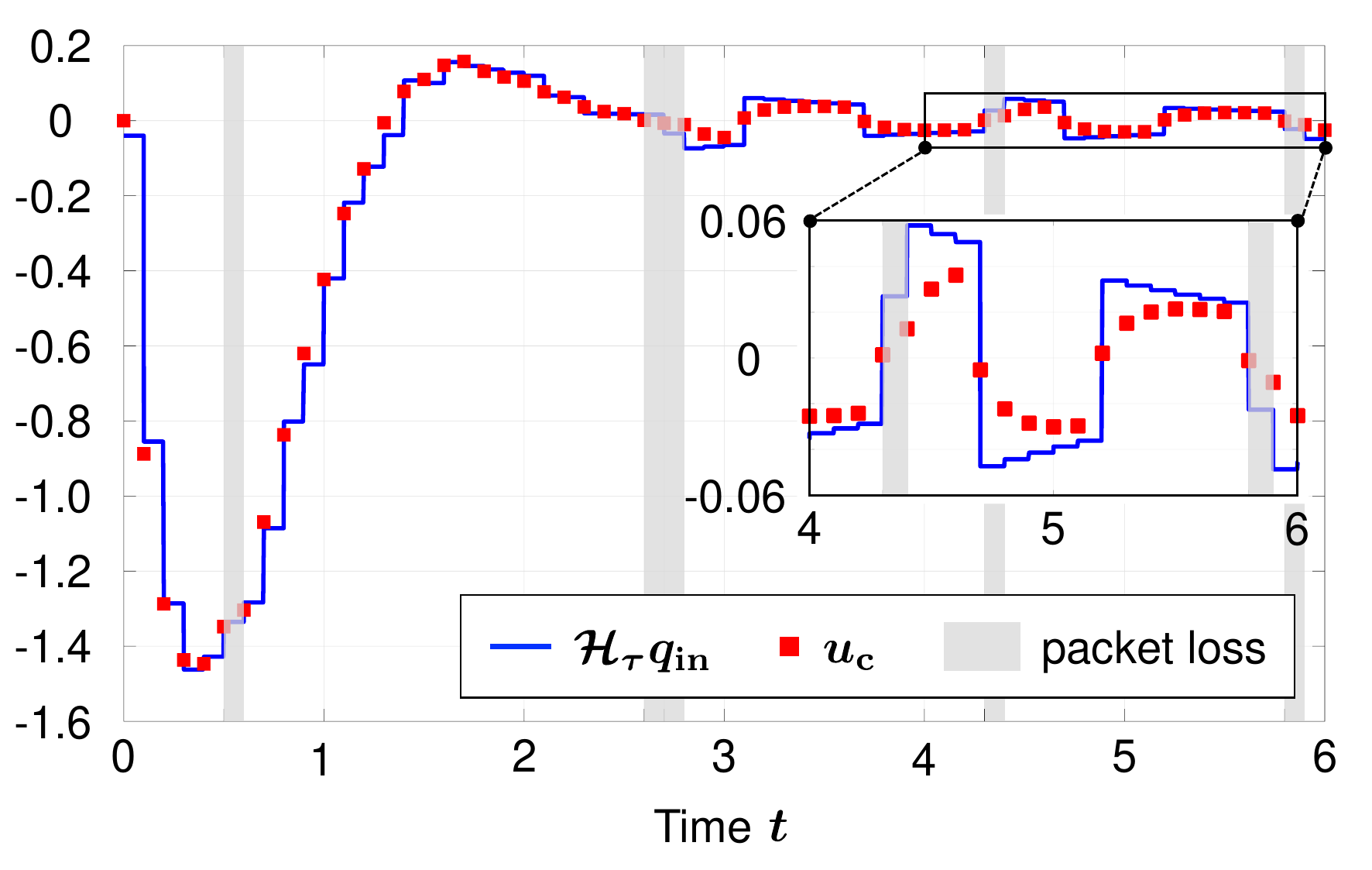}}
	\caption{Time responses under hold compensation.
		\label{fig:hold_time_response}}
\end{figure}

\section{Concluding remarks}
\label{sec:conclusion}
We have studied the output feedback stabilization problem for
infinite-dimensional systems with quantization and packet loss. First,
we have proposed the design method of dynamic quantizers for discrete-time systems with
packet loss
and have provided a sufficient condition for the closed-loop
system to achieve exponential convergence.
Next, 
we have presented an approximation-based method for computing
the  operator norms used in the quantizer design.
Finally, 
for
the sampled-data regular linear system with quantization and packet loss,
we have proved that exponential convergence at sampling times
leads to that on $\mathbb{R}_+$. 
We have also provided
a series representation of the feedthrough matrix
of the sampled-data system, by using
the spectral expansion of the semigroup generator of
the plant.
Future work will focus on addressing quantized control of systems
with weaker stability properties
such as semi-uniform stability.

\section*{Acknowledgments}
The author would like to thank Takehiko Kinoshita, 
Yoshitaka Watanabe, and Mitsuhiro T. Nakao  for sharing
the unpublished note related to \cite[Theorem~2.9]{Kinoshita2020}, which 
provided the core idea for proving Theorem~\ref{thm:norm_conv}.

\end{document}